\documentclass[final]{paper}

\usepackage[bottom]{footmisc}

\usepackage{enumitem,centernot,comment}
\usepackage{epic,eepic}
\usepackage{graphicx,color}
\usepackage[color,notref,notcite]{showkeys}
\usepackage{graphicx}
\usepackage{url}
\usepackage{amssymb,amsmath}
\usepackage{colortbl}
\usepackage{multicol}
\usepackage{multirow}
\usepackage{bm}
\usepackage{makeidx}
\usepackage{subfigure}
\usepackage{framed}
\usepackage[most]{tcolorbox}
\usepackage{ifthen}
\usepackage{thmtools,thm-restate}

\makeindex             
\allowdisplaybreaks

\definecolor{labelkey}{rgb}{0.6,0,1}

\usepackage[normalem]{ulem}
\usepackage{bm,marginnote}
\newcounter{corr}
\definecolor{violet}{rgb}{0.580,0.,0.827}
\newcommand{\corr}[3]{\typeout{Warning : a correction remains in page
\thepage}
				\stepcounter{corr}        
				{\color{blue}\ifmmode\text{\,\sout{\ensuremath{#1}}\,}\else\sout{#1}\fi}
        {\color{red}#2}
        {\color{violet} \fbox{\thecorr}#3} \corrbox}

\def\corrbox{\marginnote{\color{red}{\bf CORRBOX}}}

\newif\ifcompil\compilfalse

\newcounter{cst}
\catcode`\@=11
\def\ctel#1{C_{\refstepcounter{cst}\@bsphack
\protected@write\@auxout{}%
           {\string\newlabel{#1}{{\thecst}{\thepage}}}\thecst}}
\catcode`\@=12

\newcommand{\cter}[1]{C_{\ref{#1}}}

\newcounter{cexp}
\catcode`\@=11
\def\terml#1{T_{\refstepcounter{cexp}\@bsphack
\protected@write\@auxout{}%
           {\string\newlabel{#1}{{\thecexp}{\thepage}}}\thecexp}}
\catcode`\@=12

\newcommand{\mathbi}[1]{{\boldsymbol #1}}

\newcommand{\eop}{{\unskip\nobreak\hfil\penalty50
           \hskip2em\hbox{}\nobreak\hfil\mbox{\rule{1ex}{1ex} \qquad}
   \parfillskip=0pt
   \finalhyphendemerits=0\par\medskip}}
\newenvironment{proof}[1][]{\noindent {\bf Proof#1. } }{\eop}

\newtheorem{theorem}{Theorem}[section]
\newtheorem{remark}[theorem]{Remark}

\newtheorem{lemma}[theorem]{Lemma} 
\newtheorem{definition}[theorem]{Definition}

\numberwithin{equation}{section}

\definecolor{shadecolor}{gray}{0.92}
\definecolor{TFFrameColor}{gray}{0.92}
\definecolor{TFTitleColor}{rgb}{0,0,0}
\definecolor{ColourPropGS}{rgb}{.68,.88,.98}

\newcommand{\ba}{\begin{array}{llll}   }
\newcommand{\bac}{\begin{array}{c}}
\newcommand{\bari}{\begin{array}{r}}
\newcommand{\ea}{\end{array}}

\newcommand{\ban}{\begin{array}{llll}}
\newcommand{\ean}{\end{array}}

\newcommand{\be}{\begin{equation}}
\newcommand{\ee}{\end{equation}}

\newcommand{\beqsys }{\beqtab \left \{ \begin{array}{l}}
\newcommand{\eeqsys }{\end{array} \right . \eeqtab }

\newcommand{\benum}{\begin{enumerate}}
\newcommand{\eenum}{\end{enumerate}}

\newcommand{\beqtab}{\begin{eqnarray}} 
\newcommand{\eeqtab}{\end{eqnarray}}

\newcommand{\dsp}{\displaystyle}

\newcommand{\basex}{\xi}        

\newcommand{\bfa}{\mathbi{a}}
\newcommand{\bfA}{\mathbi{A}}

\newcommand{\bfn}{\mathbi{n}}

\newcommand{\bv}{\mathbi{v}}

\newcommand{\bF}{\mathbi{F}}

\newcommand{\bvarphi}{\mathbi{\varphi}}
\newcommand{\bpsi}{\mathbi{\psi}}
\newcommand{\bu}{\ubarre}
\newcommand{\bw}{\mathbi{w}}

\renewcommand{\d}{{\rm d}}

\newcommand{\dfrontiere}{\d\gamma}

\newcommand{\disc}{{\mathcal D}}

\newcommand{\dr}{\partial}

\renewcommand{\div}{{\mathop{\rm div}}}

\newcommand{\eps}{\varepsilon}

\newcommand{\grad}{\nabla}

\newcommand{\half}{{\frac 1 2}}

\newcommand{\Idisc}[1][]{I_{\disc_{#1}}}

\newcommand{\tr}{{\gamma}}
\newcommand{\trrec}{\mathbb{T}} 

\newcommand{\Wdivpprime}{W_{\! \div}^{p'}}
\newcommand{\Wdivpprimepart}{W_{\! \!  \div,\partial}^{p'}}
\newcommand{\Wdivpprimezero}{W_{\! \!  \div,0}^{p'}}

\newcommand{\mnn}{{m \in \N}}

\newcommand{\wunp}{W_\agrad}
\newcommand{\wunpdense}{\widetilde{W}_\agrad}
\newcommand{\lp}{L}
\newcommand{\lppzero}{V}
\newcommand{\lpd}{\mathbi{\lp}}
\newcommand{\hdiv}{\mathbi{W}_\adiv}
\newcommand{\hdivdense}{\widetilde{\mathbi{W}}_\adiv}
\newcommand{\biu}{{\mathbi{u}}}
\newcommand{\adiv}{\textsc{D}}
\newcommand{\agrad}{\textsc{G}}
\newcommand{\api}{\textsc{P}}

\newcommand{\N}{\mathbb N}

\newcommand{\norm}[2]{\| #1 \|_{#2}}

\newcommand{\ntr}{{\gamma_{\mathbf{n}}}}

\renewcommand{\O}{\Omega}

\renewcommand{\phi}{\varphi}

\newcommand{\R}{\mathbb R}

\newcommand{\snorm}[2]{\left\vert #1 \right\vert_{#2}}

\newcommand{\ubarre}{{\overline u}}

\newcommand{\vbarre}{{\overline v}}

\newcommand{\x}{\mathbi{x}}

\def\argmin{\mathop{\,\rm argmin\;}}

\renewcommand{\norm}[2]{\left\Vert#1\right\Vert_{#2}}

\usepackage{xparse}
\DeclareDocumentCommand{\RPiD}{ O{\disc} O{,0} }{\api_{#1}(X_{#1#2})}
\DeclareDocumentCommand{\RPiDm}{ O{\disc_m} O{,0} }{\api_{#1}(X_{#1#2})}

\includecomment{commentforus}

\begin{document}

\title{A unified analysis of elliptic problems with various boundary conditions and their approximation}

\author{J. Droniou, R. Eymard, T. Gallou\"et and R. Herbin}

\maketitle

\abstract{We design an abstract setting for the approximation in Banach spaces of operators acting in duality. 
A typical example are the gradient and divergence operators in Lebesgue--Sobolev spaces on a bounded domain. 
We apply this abstract setting to the numerical approximation of Leray-Lions type problems, which include in particular linear diffusion. 
The main interest of the abstract setting is to provide a unified convergence analysis that simultaneously covers 
\begin{itemize}
	\item[(i)] all usual boundary conditions, 
	\item[(ii)] several approximation methods. 
\end{itemize}
The considered approximations can be conforming, or not (that is, the approximation functions can belong to the energy space of the problem, or not), and include classical as well as recent numerical schemes. Convergence results and error estimates are given. 
We finally briefly show how the abstract setting can also be applied to other models, including flows in fractured medium, elasticity equations and diffusion equations on manifolds.

A by-product of the analysis is an apparently novel result on the equivalence between general Poincar\'e inequalities and the surjectivity of the divergence operator in appropriate spaces.
}

\section{Introduction}\label{sec:intro}

We are interested in the approximation of linear and non-linear elliptic with various boundary conditions. 

Numerical schemes for the approximation of nonlinear diffusion problems of Leray-Lions type on standard meshes have already been studied proposed and studied in the past. 
Finite elements were proposed  for the particular case of the $p$-Laplace problem \cite{barrett4,barrett5,barrett6,barrett7}  as well as for quasi-linear problems and  models of Non-Newtonian models in glaciology \cite{glo-03-app,barrett8}. 
More recently, non conforming numerical schemes defined on polytopal meshes were introduced; discrete duality finite volume schemes  were studied in \cite{andreianov1,andreianov2,andreianov3,andreianov4}.
Other schemes which have been showed to be part of the gradient discretisation method reviewed in the recent book \cite{koala}, were also studied for the Leray-Lions type problems,  namely the SUSHI scheme \cite{popo}, the mixed finite volume scheme \cite{dro-06-ll}, the mimetic finite difference method \cite{ant-15-mim}; the  discontinuous Galerkin approximation was  considered in \cite{bur-08-dis,eym-17-dis} and the hybrid high order scheme in \cite{dip-17-hyb}.
In all these works, usually only one type of boundary conditions is considered (most often homogeneous Dirichlet boundary conditions).
These schemes have been shown to be part of the GDM framework in \cite[Part III]{koala}, to the convergence analysis of \cite[Part II]{koala} holds for each of them. 
However, the analysis performed therein is done for each type of boundary conditions (Dirichlet, Neumann, Fourier). 
Our aim  here is to provide a unified formulation of the continuous and discrete problems that covers all boundary conditions; this formulation is based on some abstract Banach spaces in which both the continuous and approximate problems are posed.

This paper is organised as follows. 
The next section is devoted to an illustrative example, which shows how to build the abstract spaces and operators in order to express a variety of problems with a variety of boundary conditions. 
In Section \ref{sec:abstract.cont}, we provide the detailed framework concerning the function spaces, and the core properties of the Gradient Discretisation Method. 
In Section \ref{sec:absleraylions}, we apply this framework to the approximation of an abstract Leray-Lions problem, and we prove the convergence of the approximation methods. 
Then we turn in Section \ref{sec:linell} to the approximation of a linear elliptic problem, deduced from the abstract Leray-Lions problem, with $p=2$. 
Note that in this case the framework becomes Hilbertian. 
Finally, in Section \ref{sec:otherapp}, we briefly review a series of applications of the unified discretisation setting.

\section{An illustrative example}\label{sec:illust}
In this section,  we take $p\in (1,+\infty)$ and define $p'\in (1,+\infty)$ by $1/p+1/p' = 1$, and consider
an archetypal example of elliptic problems, that is the anisotropic $p$-Laplace problem, which reads:
\be
  -\div (\Lambda |\grad\ubarre|^{p-2}\grad\ubarre) = r + \div\bF \hbox{ in } \Omega,
  \label{pbintro}\ee
where
\begin{subequations}
\begin{align}
  \bullet~ & \O \mbox{ is an open bounded connected subset of $\R^d$ ($d\in\N^\star$) with boundary } \partial\Omega, \label{hypomegaintro} \\
\bullet~ & \Lambda \hbox{ is a measurable function from } \Omega  \hbox{ to the set of $d\times d$ symmetric matrices, }\nonumber\\
&\mbox{and there exists $\underline{\lambda},\overline{\lambda}>0$ such that, for a.e.\ $\x\in\O$, $\Lambda(\x)$ has eigenvalues in $[\underline{\lambda},\overline{\lambda}]$,}
\label{hyplambdaintro}\\
\bullet~  & r \in L^{p'}(\O)\hbox{ and }\bF \in L^{p'}(\O)^d. \label{hypfgintro}
\end{align}
\label{hypgnlintro}
\end{subequations}
This problem can be considered with a variety of boundary conditions (BCS), with an additional condition on $\ubarre$ in the case of Neumann boundary conditions. 
These conditions are summarised in Table \ref{tab:bcs}, in which $\bfn$ denotes the outer normal to $\dr\O$.

\begin{table}[htb]
\resizebox{\textwidth}{!}{
\begin{tabular}{|c|c|c|c|c|}
 \hline
  & \begin{tabular}{c}homogeneous\\Dirichlet\end{tabular} & \begin{tabular}{c}homogeneous\\Neumann \end{tabular}& \begin{tabular}{c}non-homogeneous\\ Neumann \end{tabular} & Fourier \\
 \hline
 on $\dr\O$ & $\ubarre=0$ &  $\ba(\Lambda|\grad\ubarre|^{p-2}\grad\ubarre+\bF)\cdot\bfn\\
 =0\ea$ & $\ba(\Lambda|\grad\ubarre|^{p-2}\grad\ubarre+\bF)\cdot\bfn\\ =g\ea$  & $\ba(\Lambda|\grad\ubarre|^{p-2}\grad\ubarre+\bF)\cdot\bfn \\
 + b |\ubarre|^{p-2}\ubarre =g\ea$ \\
 \hline
 \multirow{2}{*}{
\begin{tabular}{c}
additional \\conditions\end{tabular}
} &  &  $\int_\O r(\x)\d\x = 0$  & \begin{tabular}{c}$g\in L^{p'}(\dr\O)$ \\
 $\ba \int_\O r(\x)\d\x \\ + \int_{\dr\O} g(\x)\dfrontiere(\x)=0\ea$\end{tabular} \raisebox{-2em}{\rule{0em}{4.3em}}&  \begin{tabular}{c}$g\in L^{p'}(\dr\O)$ \\ $b\in L^\infty(\dr\O)$ 
 \\ $0 < \underline{b} \le b(\x)$ \end{tabular}\\
 \cline{2-5}
&  &  $\int_\O \ubarre(\x)\d\x = 0$ \raisebox{-.5em}{\rule{0em}{1.6em}} &  $\int_\O \ubarre(\x)\d\x = 0$ &  \\
 \hline
\end{tabular}
}
\caption{Various boundary conditions for \eqref{pbintro}.}
\label{tab:bcs}
\end{table}

The analysis of approximations of \eqref{pbintro} can then be carried out, for each of these boundary conditions; a usual way is to first write a weak formulation of the problem and then design tools to approximate this formulation. 
For non-homogeneous Neumann BCs and Fourier BCs, these tools must include the approximation of the trace on the boundary. 
Let us now describe a unified formulation of \eqref{pbintro}  that includes all considered boundary conditions, together with a generic approximation scheme based on this unified formulation. 

\medskip

Introduce two Banach spaces  $\lpd = L^p(\O)^d$ and $\lp$, a space $\wunp\subset \lp$ (which is dense in $\lp$), an operator $\agrad:\wunp\to\lpd$, two mappings $\bfa:\lp\times\lpd\to\lpd'$ and $a:\lp\to \lp'$ and a right-hand-side $f\in \lp'$ as in Table \ref{tab:operators}. 
Here and in the rest of the paper, $\tr u$ is the trace on $\dr\O$ of any function $u\in W^{1,p}(\O)$.

\begin{table}[htb]
\resizebox{\textwidth}{!}{
\begin{tabular}{|c|c|c|c|c|}
 \hline
  & \begin{tabular}{c}homogeneous\\Dirichlet\end{tabular} & \begin{tabular}{c}homogeneous\\Neumann \end{tabular}& \begin{tabular}{c}non-homogeneous\\ Neumann \end{tabular} & Fourier \\
 \hline
 $\lp=$ & $L^p(\O)$ &  $L^p(\O)$ & $L^p(\O)\times L^p(\dr\O)$ & $L^p(\O)\times L^p(\dr\O)$ \\
 \hline
 $\wunp=$ & $W^{1,p}_0(\O)$ &  $W^{1,p}(\O)$ & $ \{(u,\tr u)\,:\,u\in W^{1,p}(\O) \}$  & $\{(u,\tr u)\,:\,u\in W^{1,p}(\O) \}$  \\
 \hline
 $\agrad:$ & $u\mapsto\nabla u$ &  $u\mapsto\nabla u$ & $ (u,w)\mapsto\nabla u$  & $(u,w)\mapsto\nabla u$ \\
 \hline
 $\bfa:$ & $\ba (u,\bv)\mapsto\\ \quad
 \Lambda|\bv|^{p-2} \bv\ea$ &  $\ba (u,\bv)\mapsto\\ \quad
 \Lambda|\bv|^{p-2} \bv\ea$ & $\ba ((u,w),\bv)\mapsto\\ \quad
 \Lambda|\bv|^{p-2} \bv\ea$  & $\ba ((u,w),\bv)\mapsto\\ \quad
 \Lambda|\bv|^{p-2} \bv\ea$ \\
 \hline
 $a:$ & $u\mapsto 0$ &  $\ba u\mapsto \\
 |\int_\O u|^{p-2}(\int_\O u) 1_\Omega \ea$ \raisebox{-1.2em}{\rule{0em}{2.3em}}& $\ba (u,w)\mapsto\\
 |\int_\O u|^{p-2}(\int_\O u) (1_\Omega,0)\ea$  & $\ba(u,w)\mapsto\\
 (0,b |\tr u|^{p-2}\tr u)\ea$ \\
 \hline
 $f=$ & $r$ &  $r$ & $(r,g)$  & $(r,g)$ \\
 \hline
\end{tabular}
}
\caption{Abstract operators for various boundary conditions.}
\label{tab:operators}
\end{table}

The weak formulation of Problem \eqref{pbintro} with all considered BCs is then:
\be
\begin{aligned}
&\mbox{Find $\ubarre\in\wunp$  such that, }\forall v\in\wunp,\\
&\langle\bfa(\ubarre,\agrad \ubarre),\agrad v \rangle_{\lpd',\lpd} + \langle a(\ubarre),v \rangle_{\lp',\lp}  = \langle f,v \rangle_{\lp',\lp} -\langle \bF,\agrad v \rangle_{\lpd',\lpd}.
\end{aligned}
\label{eq:weakintro}\ee

Indeed, in the case of homogeneous Dirichlet boundary conditions, $\norm{\nabla\cdot}{L^p(\O)^d}$ is a norm on the space $\wunp = W^{1,p}_0(\O)$ (owing to Poincar\'e's inequality) and there is no need for an additional condition: we can then let $a=0$. 

In the case of homogeneous Neumann conditions, multiplying \eqref{pbintro} by $v = 1_\O$ and integrating over $\O$ shows that the condition  
$\int_\O r(\x)\d\x = 0$ is necessary for the existence of at least one solution; this solution
is defined up to an additive constant which is fixed by imposing, for example, $\int_\O \ubarre(\x)\d\x = 0$. A classical technique to write a weak formulation that embeds this condition, and has
the required coercivity property, is to introduce an additional term $\langle a(\ubarre),v \rangle_{\lp',\lp}$ in the left-hand side of this formulation, where $ a(\ubarre) = |\int_\O u|^{p-2}\int_\O u 1_\Omega $.
Non-homogeneous Neumann BCs are handled in a similar way.

For Fourier boundary conditions, the term $\langle a(\ubarre),v \rangle_{\lp',\lp} = \int_{\dr\O} b |\tr u|^{p-2}\tr u \tr v \dfrontiere$ naturally appears when multiplying \eqref{pbintro} by a test function $v$ and
formally integrating by parts.

\medskip

Problem \eqref{eq:weakintro} can be re-formulated by introducing a space $\hdiv\subset \lpd'$ and the dual operator $\adiv:\hdiv\to \lp'$ to $\agrad$ as per Table \ref{tab:dual}.
In this table, we set $\Wdivpprime(\O) = \{\bvarphi\in L^{p'}(\O)^d\,:\, \div\bvarphi\in L^{p'}(\O)\}$ $\Wdivpprimezero(\O)=\{\bvarphi\in \Wdivpprime(\O)\,:\,\ntr\bvarphi=0\}$ and $\Wdivpprimepart(\O) = \{\bvarphi\in \Wdivpprime(\O)\,:\, \ntr\bvarphi\in L^{p'}(\partial\O)\}$, where $\ntr\bvarphi$ is the normal trace of $\bv$ on $\partial\O$.
The space $\hdiv$ and operator $\adiv$ are defined such that the following formula, which generalises the Stokes formula to all type of boundary conditions, holds:
\begin{equation}
 \forall u\in\wunp,\ \forall \bv\in\hdiv,\ \langle \bv,\agrad u\rangle_{{\lpd}',\lpd} + \langle \adiv\bv,u\rangle_{\lp',\lp} = 0.\label{abs:stokes-formulaintro}
\end{equation}

\begin{table}[!h]
\begin{center}
\begin{tabular}{|c|c|c|c|c|}
 \hline
  & \begin{tabular}{c}homogeneous\\Dirichlet\end{tabular} & \begin{tabular}{c}homogeneous\\Neumann \end{tabular}& \begin{tabular}{c}non-homogeneous\\ Neumann \end{tabular} & Fourier \\
 \hline
 $\hdiv =$ & $\Wdivpprime(\O)$ &  $\Wdivpprimezero(\O)$ & $ \Wdivpprimepart(\O)$  & $\Wdivpprimepart(\O)$ \\
 \hline
 $\adiv:$ & $\bv\mapsto\div \bv$ &  $\bv\mapsto\div \bv$ & $\bv\mapsto(\div\bv,-\ntr\bv)$  & $\bv\mapsto(\div\bv,-\ntr\bv)$ \\
 \hline
\end{tabular}
\caption{Dual space and operators for various boundary conditions.}
\label{tab:dual}
\end{center}
\end{table}

 Problem \eqref{eq:weakintro}  is then equivalent to
\be
	\begin{aligned}
		&\mbox{Find $\ubarre\in\wunp$  such that $\bfa(\ubarre,\agrad \ubarre) +\bF\in\hdiv$ and } - \adiv \big(\bfa(\ubarre,\agrad \ubarre)+\bF\big) + a(\ubarre) = f \mbox{ in } \lp'.
	\label{eq:pbstrongintro} 
\end{aligned}
\ee

This equivalence is proved in Section \ref{sec:absleraylions} in the general abstract setting.
Thanks to the above introduced framework, approximations of Problem \eqref{eq:weakintro} can be designed by drawing inspiration from the Gradient Discretisation Method (GDM), see \cite{koala}. 
Three discrete objects $\disc=(X_\disc,\api_\disc,\agrad_\disc)$, forming altogether a \emph{gradient discretisation}, are introduced: a finite dimensional vector space $X_{\disc}$ meant to contain the families of discrete unknowns, a linear mapping $\api_\disc~:~X_{\disc}\to \lp$ that reconstructs an element in $\lp$ from an element of $X_{\disc}$, and a ``gradient'' reconstruction $\agrad_\disc~:~X_{\disc}\to \lpd$, which is a linear mapping that reconstructs an element in $\lpd$ from an element of $X_{\disc}$. 
The \emph{gradient scheme} for the approximation of Problem \eqref{eq:weakintro} is then obtained by replacing the continuous space and operators by the discrete ones:
\be
\begin{aligned}
&\mbox{Find $u\in X_\disc$ such that, } \forall v\in X_\disc\,,\\
&\langle\bfa(\api_\disc u, \agrad_\disc u),\agrad_\disc v \rangle_{\lpd',\lpd} + \langle a(\api_\disc u ),\api_\disc v \rangle_{\lp',\lp}   = \langle f,\api_\disc v \rangle_{\lp',\lp} -\langle \bF,\agrad_\disc v \rangle_{\lpd',\lpd}.
\end{aligned}
\label{gradsch_genintro}\ee

Note that $P_\disc$ denotes either a reconstructed function over $\O$ (Dirichlet or homogeneous Neumann) conditions, or a pair reconstructed function on $\Omega$ and reconstructed  trace on $\partial \Omega$ (homogeneous Neumann and Fourier conditions, see Table \ref{tab:PiDtrrec}).

\begin{table}[!h]
\begin{center}
\begin{tabular}{|c|c|c|c|c|}
 \hline
  & \begin{tabular}{c}homogeneous\\Dirichlet\end{tabular} & \begin{tabular}{c}homogeneous\\Neumann \end{tabular}& \begin{tabular}{c}non-homogeneous\\ Neumann \end{tabular} & Fourier \\
 \hline
 $\api_\disc:$ & $u\mapsto\Pi_\disc u$ &  $u\mapsto\Pi_\disc u$ & $u\mapsto(\Pi_\disc u,\trrec_\disc u)$  & $u\mapsto(\Pi_\disc u,\trrec_\disc u)$ \\
 \hline
\end{tabular}
\end{center}
\caption{Function ($\Pi_\disc$) and trace ($\trrec_\disc$) reconstructions for various boundary conditions.}
\label{tab:PiDtrrec}
\end{table}


We now generalise this process from the continuous problems to the discrete ones in the remaining part of this paper.

\section{Continuous and discrete settings}\label{sec:abstract.cont}

The examples in Section \ref{sec:illust} gave a flavour of a general setting we now describe.

\subsection{Continuous spaces and operators}\label{sec:cont.setting}

Let $\lp$ and $\lpd$ be separable reflexive Banach spaces, with respective topological dual spaces  $\lp'$ and $\lpd'$. 
Let $\wunp\subset\lp$ be a dense subspace of $\lp$ and let $\agrad:\wunp\to \lpd$ be a linear operator whose graph $\mathcal{G} = \{ (u,\agrad u), u\in \wunp\}$ is closed in $\lp\times\lpd$. 
As a consequence, $\wunp$ endowed with the graph norm $\norm{u}{\wunp,\mathcal{G}} = (\norm{u}{\lp}^p + \norm{\agrad u}{\lpd}^p)^{1/p}$ is a Banach space continuously embedded in $\lp$.  
Since $\lp\times\lpd$ is separable, $\wunp$ is also separable for the norm $\norm{\cdot}{\wunp,\mathcal{G}}$ (see  \cite[Ch. III]{brezis}).

\medskip

Define $\hdiv$ by:
\be
\hdiv = \{ \bv\in \lpd'\,:\, \exists w\in\lp', \forall u\in\wunp, \langle \bv,\agrad u\rangle_{{\lpd}',\lpd} + \langle w,u\rangle_{\lp',\lp} = 0\}.
\label{abs:defhdiv}\ee
The density of $\wunp$ in $\lp$ implies (and is actually equivalent to) the following property.
\be\label{wunp.prop}
\mbox{For all $w\in\lp'$, }(\forall u\in\wunp,  \langle w,u\rangle_{{\lp}',\lp} = 0)\Rightarrow w = 0.
\ee
Therefore, for any $\bv\in\hdiv$, the element $w\in\lp'$ whose existence is assumed in \eqref{abs:defhdiv} is unique; this defines a linear operator $\adiv:\hdiv\to\lp'$, adjoint operator of $-\agrad$ in the sense of \cite[p.167]{Kato1995}, such that $w = \adiv \bv$, that is,
\begin{equation}
 \forall u\in\wunp,\ \forall \bv\in\hdiv,\ \langle \bv,\agrad u\rangle_{{\lpd}',\lpd} + \langle \adiv\bv,u\rangle_{\lp',\lp} = 0.\label{abs:stokes-formula}
\end{equation}
It easily follows from this that the graph of $\adiv$ is closed in $\lpd'\times \lp'$, and therefore that, endowed with the graph norm $\norm{\bv}{\hdiv} = (\norm{\bv}{{\lpd}'}^{p'} + \norm{\adiv\bv}{\lp'}^{p'})^{1/p'}$, $\hdiv$ is a Banach space continuously embedded and dense in $\lpd'$ (see \cite[Theorem 5.29 p.168]{Kato1995}).

\begin{remark}[Reverse construction of the dual operators]\label{rem:reverse}
Since the spaces $\lp$ and $\lpd$ are reflexive, \cite[Theorem 5.29 p.168]{Kato1995} also states that there holds
\begin{align*}
 &\wunp = \{ u\in \lp\,:\, \exists \biu\in\lpd, \forall \bv\in\hdiv, \langle \bv,\biu\rangle_{{\lpd}',\lpd} + \langle \adiv\bv,u\rangle_{\lp',\lp} = 0\},\\
&\forall u\in\wunp\,,\; \agrad u=\mbox{the element $\biu\in\lpd$ in the definition of $\wunp$}.
\end{align*}
It is therefore equivalent to begin with  the construction of $(\wunp,\agrad)$ or that of $(\hdiv,\adiv)$.
\end{remark}

Let $\lppzero$ be a closed subspace of $\lp'$ and denote by $\snorm{\cdot}{\lp,\lppzero}$  the semi-norm on $\lp$ defined by
\be
\forall u\in \lp, \ \snorm{u}{\lp,\lppzero} = \left\{\begin{array}{cl}
\dsp \sup_{\mu\in \lppzero\setminus\{0\} } \frac {|\langle \mu, u\rangle_{\lp',\lp}|} {\norm{\mu}{\lp'}}&\mbox{ if $\lppzero\not=\{0\}$},\\[1em]
0&\mbox{ if $\lppzero=\{0\}$.}
\end{array}
\right.
\label{abs:defseminorm}\ee
By construction, there holds, for all $u\in\lp$, $\snorm{u}{\lp,\lppzero} \le \sup_{\mu\in \lp'\setminus\{0\} } \frac {|\langle \mu, u\rangle_{\lp',\lp}|} {\norm{\mu}{\lp'}}=\norm{u}{\lp}$.
Defining, for $u\in \wunp$,
\be
\norm{u}{\wunp} = (\snorm{u}{\lp,\lppzero}^p + \norm{\agrad u}{\lpd}^p)^{1/p},
\label{abs:defnormwunp}\ee
we therefore have
\be\label{un.sens}
\forall u\in \wunp\,,\quad\norm{u}{\wunp}\le \norm{u}{\wunp, \mathcal{G}}.
\ee
In the following, we assume that these semi-norms and norms are actually equivalent,
that is, there exists $C_{\wunp,\lppzero}>0$ such that
\be
\forall u\in\wunp, \quad\norm{u}{\wunp, \mathcal{G}}\le  C_{\wunp,\lppzero} \norm{u}{\wunp}.
\label{eq:hypV}\ee

A necessary and sufficient condition on $V$ for \eqref{eq:hypV} to hold is that $\lp'={\rm Im}(\adiv)+\lppzero$ as stated in the following theorem, whose proof is based on a Galerkin-type method which enters the gradient discretisation method framework introduced in Section \ref{sec:abstract.cont} applied to the abstract Leray-Lions problem of Section \ref{sec:absleraylions}.

\begin{theorem}\label{abs:thm:normeq}
Assume the setting described in Section \ref{sec:cont.setting},
except \eqref{eq:hypV}.
Then
\be
\lp'={\rm Im}(\adiv)+\lppzero
\label{abs:suppl}\ee
if and only if $ \norm{\cdot}{\wunp} $ and $\norm{\cdot}{\wunp,\mathcal{G}} $  are two equivalent norms (that is, \eqref{eq:hypV} holds).
\end{theorem}

The proof is given in  Section \ref{sec:absleraylions}.

\begin{remark}[Poincar\'e inequalities] 
In the particular context of Sobolev spaces, Theorem \ref{abs:thm:normeq} proves that there is equivalence between the so-called ``mean'' Poincar\'e--Wirtinger inequality and the surjectivity of the divergence operator. 
\end{remark}

\subsection{Gradient discretisations}

Based on the previous definitions, we generalise the concept of gradient discretisation of \cite{koala} and the key notions of coercivity, limit-conformity, consistency and compactness to the present abstract setting.
These properties enable us, in Section \ref{sec:absleraylions}, to design converging approximation schemes for an abstract monotonous problem. 

\subsubsection{Key definitions}

\begin{definition}[Gradient Discretisation]
\label{def:graddisc}
In the setting described in Section \ref{sec:cont.setting}, a gradient discretisation is defined by $\disc = (X_{\disc},\api_\disc,\agrad_\disc)$, where:
\begin{enumerate}
\item The set of discrete unknowns 
$X_{\disc}$ is a finite dimensional vector space on $\R$.
\item The ``function'' reconstruction $\api_\disc~:~X_{\disc}\to \lp$ is a linear mapping that reconstructs, from an element of $X_{\disc}$, an element in $\lp$.
\item The ``gradient'' reconstruction $\agrad_\disc~:~X_{\disc}\to \lpd$ is a linear mapping that reconstructs, from an ele\-ment of $X_{\disc}$, an element of $\lpd$.
\item The mappings $\api_\disc$ and $\agrad_\disc$ are such that the following quantity is a norm on $X_\disc$:
\[
\norm{v}{\disc} := \left(\snorm{\api_\disc v}{\lp,\lppzero}^p + \norm{\agrad_\disc v}{\lpd}^p\right)^{1/p}.
\]
\end{enumerate}
\end{definition}

\begin{definition}[Coercivity] \label{abs:def-coer}
If $\disc$ is a gradient discretisation in the sense of Definition \ref{def:graddisc}, let $C_\disc$ be the norm of $\api_\disc$:
\be
C_\disc =  \max_{v\in X_{\disc}\setminus\{0\}}\frac {\norm{\api_\disc v}{\lp }} {\Vert v \Vert_{\disc}}.
\label{abs:defcoercivity}\ee 
A sequence $(\disc_m)_{m\in\N}$ of gradient discretisations is \textbf{coercive} if  there exists $ C_P \in \R_+$ such that $C_{\disc_m}\le C_P$ for all $\mnn$.
\end{definition}

\begin{definition}[Limit-conformity] \label{abs:def-limconf}
If $\disc$ is a gradient discretisation in the sense of Definition \ref{def:graddisc}, 
let $W_{\disc}:\hdiv \to [0,+\infty)$ be given by
\be
\forall \bvarphi\in \hdiv \,,\;
W_{\disc}(\bvarphi) = \sup_{u\in X_{\disc}\setminus\{0\}}\frac{\dsp
\left|\langle \bvarphi,\agrad_\disc u\rangle_{{\lpd}',\lpd} + \langle \adiv\bvarphi,\api_\disc u\rangle_{\lp',\lp}
\right|}{\Vert  u \Vert_{\disc}}.
\label{abs:defwdisc}\ee
A sequence $(\disc_m)_{m\in\N}$ of gradient discretisations is
\textbf{limit-conforming} if
\be
\forall \bvarphi\in \hdiv\,,\ \lim_{m\to\infty} W_{\disc_m}(\bvarphi) = 0.
\label{abs:limconf}\ee
\end{definition}

Once $\lp$, $\lpd$, $\hdiv$ and $\adiv$ are chosen, the definition \ref{abs:def-limconf} of limit-conformity is constrained by the continuous duality formula \eqref{abs:stokes-formula}; as a consequence of Lemma \ref{abs:lcsuf-coer} below, the definition of coercivity is also constrained by this formula. 
These two notions therefore naturally follow from the continuous setting.
On the contrary, the following two definitions of consistency and compactness are disconnected from the duality formula. 
Various choices for these notions are possible, we describe here one that is in particular adapted to the monotonous problem in Section \ref{sec:absleraylions}.

\begin{definition}[Consistency] \label{abs:def-cons}
If $\disc$ is a gradient discretisation in the sense of Definition \ref{def:graddisc}, let
 $S_{\disc}:\wunp \to [0,+\infty)$ be given by
\be
\forall \varphi\in \wunp \,,\quad
S_{\disc}(\varphi) = \min_{v\in X_{\disc}}\Bigl(\norm{\api_\disc v - \varphi}{\lp } + \norm{\agrad_\disc v-\agrad\varphi}{\lpd}\Bigr).
\label{abs:defsdisc}\ee
A sequence $(\disc_m)_{m\in\N}$ of gradient discretisations
is \textbf{consistent} if
\be
\forall \varphi\in \wunp\,,\ \lim_{m\to\infty} S_{\disc_m}(\varphi)=0.
\label{abs:strongconsist}
\ee
\end{definition}

\begin{definition}[Compactness] \label{abs:def-comp}
A sequence $(\disc_m)_{m\in\N}$ of  gradient discretisations in the sense of Defi\-nition \ref{def:graddisc}
is  \textbf{compact} if, for any sequence $u_m\in X_{\disc_m}$ such that $(\Vert  u_m \Vert_{\disc_m})_{m\in\N}$ is bounded, the sequence $(\api_{\disc_m}u_m)_{m\in\N}$  is relatively compact in $\lp $. 
\end{definition}

\subsubsection{Main properties}
 
 The following result uses the surjectivity of the divergence operator proven in Theorem \ref{abs:thm:normeq}.

\begin{lemma}[Limit-conformity implies coercivity]\label{abs:lcsuf-coer}
If a sequence of gradient discretisations is limit-conforming in the sense of Definition \ref{abs:def-limconf}, then it is also coercive in the sense of Definition \ref{abs:def-coer}.
\end{lemma}

\begin{proof} Consider a limit-conforming sequence $(\disc_m)_{\mnn}$ and set 
\[
E=\left\{\frac{\api_{\disc_m}v}{\norm{v}{\disc_m}}\in \lp \,:\,\mnn\,,\;v\in X_{\disc_m}\backslash\{0\}\right\}.
\] 
Proving the coercivity of $(\disc_m)_\mnn$ consists in proving that
$E$ is bounded in $\lp $. Let $f\in \lp'$. By Theorem \ref{abs:thm:normeq},
there exists $\bv_f\in\hdiv$ and $\mu_f\in \lppzero$ such that $f=\adiv \bv_f+\mu_f$.
The definition of $\snorm{\cdot}{\lp,\lppzero}$ shows that $|\langle \mu_f, \cdot\rangle_{\lp',\lp}|\le \norm{\mu_f}{\lp'}\snorm{\cdot}{\lp,\lppzero}$.
For $z\in E$, take $\mnn$ and $v\in X_{\disc_m}\backslash\{0\}$ such that
$z= \frac{\api_{\disc_m}v}{\norm{v}{\disc_m}}$ and write
\begin{align}
|\langle f,z\rangle_{\lp',\lp}|\le{}&
\frac{1}{\norm{v}{\disc_m}}\left|\langle \adiv\bv_f, \api_{\disc_m} v\rangle_{\lp',\lp}\right| 
+\frac{1}{\norm{v}{\disc_m}}\left|\langle \mu_f, \api_{\disc_m} v\rangle_{\lp',\lp}\right| \nonumber\\
\le{}&  \frac{1}{\norm{v}{\disc_m}}\left|\langle \adiv\bv_f,\api_{\disc_m}v\rangle_{\lp',\lp}+
\langle \bv_f,\agrad_{\disc_m}v\rangle_{{\lpd}',\lpd}\right|
+\frac{1}{ \norm{v}{\disc_m} }\left|\langle \bv_f,\agrad_{\disc_m}v\rangle_{{\lpd}',\lpd}\right|\nonumber\\
&+\frac{1}{\norm{v}{\disc_m}} \norm{\mu_f}{\lp'} \snorm{ \api_{\disc_m} v}{\lp,\lppzero}\nonumber\\
\le{}& W_{\disc_m}(\bv_f)+\norm{\bv_f}{\lpd'} +\norm{\mu_f}{\lp'}.
\label{abs:lc.coer.1}
\end{align}
In the last inequality we used $\snorm{\api_{\disc_m}v}{\lp,\lppzero}\le \norm{v}{\disc_m}$ and $\norm{\agrad_{\disc_m}v}{\lpd}\le
\norm{v}{\disc_m}$.
Since $(\disc_m)_{m\in\N}$ is limit-conforming, $(W_{\disc_m}(\bv_f))_\mnn$ converges to
$0$ and is therefore bounded. Estimate \eqref{abs:lc.coer.1} thus shows that
$\{\langle f,z\rangle_{\lp',\lp}\,:\,z\in E\}$ is bounded by some constant depending on $f$. Since this is valid for any $f\in \lp'$, we infer from the Banach--Steinhaus theorem \cite[Theorem 2.2]{brezis} that $E$ is bounded in $\lp $.
\end{proof}

Checking limit-conformity is made easier by the following result, which
reduces the set of elements $\bvarphi$ on which the convergence in \eqref{abs:limconf} has to be asserted.

\begin{lemma}[Equivalent condition for limit-conformity]\label{abs:suf-limconf}
A sequence $(\disc_m)_{m\in\N}$ of gradient discretisations is limit-conforming in the sense of Definition \ref{abs:def-limconf} if and only if 
it is coercive in the sense of  Definition \ref{abs:def-coer}, and there exists a dense subset $\hdivdense$ of $\hdiv $ such that
\be
\forall \bm{\psi}\in \hdivdense,\ \lim_{m\to\infty} W_{\disc_m}(\bm{\psi}) = 0.
\label{abs:limconfs}\ee
\end{lemma}

\begin{proof} If  $(\disc_m)_{m\in\N}$ is limit-conforming,
then it is coercive by Lemma \ref{abs:lcsuf-coer}, and
\eqref{abs:limconfs} is satisfied with $\hdivdense=\hdiv$ (this is \eqref{abs:limconf}).

Conversely, assume that $(\disc_m)_{\mnn}$ is coercive and that \eqref{abs:limconfs}
holds. Let $C_P \in \R_+$ be an upper bound of $(C_{\disc_m})_\mnn$. 
To prove \eqref{abs:limconf}, let $\bvarphi \in \hdiv$,
$\eps>0$ and take $\bm{\psi}\in \hdivdense$ such that $\norm{\bvarphi - \bm{\psi}}{\hdiv } \le \eps$.
By definition of the norm in $\hdiv$, this means that 
$\norm{\bvarphi - \bm{\psi}}{\lpd'} +\norm{\adiv\bvarphi - \adiv\bm{\psi}}{\lp'} \le \eps$.
Hence, for any $u\in X_{\disc_m}\backslash\{0\}$,
\begin{multline*}
\frac{\left|\langle \bvarphi-\bm{\psi},\agrad_{\disc_m} u\rangle_{{\lpd}',\lpd} + \langle \adiv\bvarphi-
\adiv\bm{\psi},\api_{\disc_m} u\rangle_{\lp',\lp}
\right|}{\norm{u}{\disc_m}}\\
\le \norm{\bvarphi-\bm{\psi}}{\lpd'}\frac{\norm{\agrad_{\disc_m} u}{\lpd}}{\norm{u}{\disc_m}}
+\norm{\adiv\bvarphi-\adiv\bm{\psi}}{\lp'}\frac{\norm{\api_{\disc_m} u}{\lp}}{\norm{u}{\disc_m}}
\le \max(1,C_P)\eps.
\end{multline*}
Introducing $\bm{\psi}$ and $\adiv\bm{\psi}$ in the definition \eqref{abs:defwdisc} of
$W_{\disc_m}(\bvarphi)$, we infer
\begin{align*}
W_{\disc_m}(\bvarphi)\le{}& \sup_{u\in X_{\disc_m}\backslash\{0\}}
\frac{|\langle \bm{\psi},\agrad_{\disc_m} u\rangle_{\lpd',\lpd}+
\langle \adiv\bpsi,\api_{\disc_m} u\rangle_{\lp',\lp}|}{\norm{u}{\disc_m}}
+ \max(1,C_P)\eps\\
={}&W_{\disc_m}(\bm{\psi})+ \max(1,C_P)\eps.
\end{align*}
Invoking \eqref{abs:limconfs} we deduce that
$\limsup_{m\to\infty}  W_{\disc_m}(\bvarphi)\le  \max(1,C_P)\eps$, and the
proof is concluded by letting $\eps\to 0$. \end{proof}

The next lemma is an essential tool to use compactness techniques in the convergence
analysis of approximation methods for non-linear models.

\begin{lemma}[Regularity of the limit]\label{abs:lem:reghunzero}
Let $(\disc_m)_{m\in\N}$ be a limit-conforming sequence of gradient discretisations, in the sense of 
Defi\-nition \ref{abs:def-limconf}. For any $\mnn$, take $u_m\in X_{\disc_m}$ and assume that $(\norm{u_m}{\disc_m})_{m\in\N}$ is bounded. Then
there exists $u\in \wunp $ such that, along a subsequence as $m\to\infty$,
$(\api_{\disc_m}u_m)_\mnn$ converges weakly in $\lp $ to $u$, and
$(\agrad_{\disc_m}u_m)_\mnn$ converges weakly in $\lpd$ to $\agrad u$.
\end{lemma}

\begin{proof} By definition of $\norm{\cdot}{\disc_m}$, $(\agrad_{\disc_m}u_m)_\mnn$ is bounded in $\lpd$.
By Lemma \ref{abs:lcsuf-coer}, $(\disc_m)_\mnn$ is coercive and therefore
$(\api_{\disc_m} u_m)_{m\in\N}$ is bounded in $\lp$. 
The reflexivity of $\lp$ and $\lpd$ thus gives a subsequence
of $(\disc_m,u_m)_{m\in\N}$, denoted in the same way, and elements $u\in \lp $ and $\mathbi{u}\in  \lpd$ such that $(\api_{\disc_m}u_m)_\mnn$ converges weakly in $\lp$ to $u$ 
and $(\agrad_{\disc_m}u_m)_\mnn$ converges weakly in $\lpd$ to $\mathbi{u}$.
These weak convergences, the limit-conformity of $(\disc_m)_{m\in\N}$ and the boundedness of $(\norm{u_m}{\disc_m})_\mnn$ 
enable us to identify the limit in \eqref{abs:limconf} to see that
\[
\forall \bvarphi\in \hdiv\,,\; \langle \bvarphi,\mathbi{u}\rangle_{{\lpd}',\lpd} + \langle \adiv\bvarphi,u\rangle_{\lp',\lp} = 0.
\]
This relation simultaneously proves that $u\in \wunp $ and that $\mathbi{u} = \agrad u$. \end{proof}

\medskip

The following is the equivalent of Lemma \ref{abs:suf-limconf} for the notion
of consistency.

\begin{lemma}[Equivalent condition for consistency]\label{abs:suf-cst}
A sequence $(\disc_m)_{m\in\N}$ of gradient discretisations is consistent in the sense of Definition \ref{abs:def-cons}
if and only if there exists a dense subset $\wunpdense$ of $\wunp $ such that
\be
  \forall \psi\in \wunpdense\,,\ \lim_{m\to\infty} S_{\disc_m}(\psi)=0.
   \label{abs:strongconsistr}
\ee
\end{lemma}

\begin{proof} Let us assume that \eqref{abs:strongconsistr} holds and let us prove \eqref{abs:strongconsist}
(the converse is straightforward, take $\wunpdense=\wunp $).
Observe first that, since $\wunp$ is continuously embedded in $\lp$, there exists $C_{\wunp}>0$ such that
\[
 \forall \varphi\in\wunp\,, \ \norm{\varphi}{\lp}\le C_{\wunp} \norm{\varphi}{\wunp}.
\]
Let $\varphi\in \wunp $. Take $\eps>0$ and $\psi\in \wunpdense$ such that $\Vert \varphi - \psi\Vert_{\wunp } \le \eps$. For $v\in X_{\disc_m}$, the triangle inequality and the definition \eqref{abs:defnormwunp} of the norm in $\wunp$ yield
\begin{align*}
\norm{\api_{\disc_m} v-\varphi}{\lp }+\norm{\agrad_{\disc_m} v-\agrad\varphi}{\lpd}
\le{}& \norm{\api_{\disc_m} v-\psi}{\lp }+\norm{\psi-\varphi}{\lp }+\norm{\agrad_{\disc_m} v-\agrad\psi}{\lpd}+ \norm{\agrad\psi-\agrad\varphi}{\lpd}\\
\le{}& \norm{\api_{\disc_m} v-\psi}{\lp }+\norm{\agrad_{\disc_m} v-\agrad\psi}{\lpd}+
(C_{\wunp}+1)\norm{\psi-\varphi}{\wunp}.
\end{align*}
Taking the infimum over $v\in X_{\disc_m}$ leads to 
$S_{\disc_m}(\varphi)\le S_{\disc_m}(\psi)+(C_{\wunp}+1)\eps$. Assumption \eqref{abs:strongconsistr} then shows that
$\limsup_{m\to\infty} S_{\disc_m}(\varphi)\le (C_{\wunp}+1)\eps$, and letting
$\eps\to 0$ concludes the proof that $S_{\disc_m}(\varphi)\to 0$ as $m\to\infty$.
\end{proof}

The next result shows that the compactness, as the limit-conformity, is stronger than the coercivity.

\begin{lemma}[Compactness implies coercivity]\label{abs:suf-coer}
If a sequence of gradient discretisations is compact in the sense of Defi\-nition \ref{abs:def-comp},
then it is coercive in the sense of Definition \ref{abs:def-coer}.
\end{lemma}

\begin{proof} Assume that $(\disc_m)_\mnn$ is compact but not coercive. Then there exists a subsequence of $(\disc_m)_{m\in\N}$ (denoted in the same way) such that, for all $m\in\N$, we can find $v_m\in X_{\disc_m}\setminus\{0\}$ satisfying
\[
\lim_{m\to\infty}\frac {\norm{\api_{\disc_m} v_m}{\lp }}{\norm{v_m}{\disc_m}} = +\infty.
\]
Setting $u_m = v_m/\norm{v_m}{\disc_m}$, this gives
$\lim_{m\to\infty} \norm{\api_{\disc_m} u_m}{\lp } = +\infty$. 
But $\norm{u_m}{\disc_m} = 1$ for all $\mnn$ and the compactness of the sequence of gradient discretisations therefore implies that $(\api_{\disc_m}u_m)_{m\in\N}$  is relatively compact in $\lp $, which is a contradiction. \end{proof}

The next two lemmas show that the compactness of $(\disc_m)_\mnn$ is strongly related to some compactness property of $\wunp$.

\begin{lemma}\label{lem:compgd}[Existence of a compact sequence of GDs implies compact embedding of $\wunp$]
 Let us assume the existence of a sequence of gradient discretisations which is consistent in the sense of Definition \ref{abs:def-cons} and compact in the sense of Definition \ref{abs:def-comp}. Then the embedding of $\wunp$ in $\lp$ is compact.
\end{lemma}
\begin{proof}
Let $(\disc_m)_\mnn$ be a consistent and compact sequence of gradient discretisations,
and $(\overline{u}_m)_{m\in\N}$ be a bounded sequence in $\wunp$. For $m=0$, let $N_0\in\N$ be such that there exists $u_{N_0}\in X_{\disc_{N_0}}$ satisfying 
 \[
  \norm{\api_{\disc_{N_0}} u_{N_0} - \overline{u}_0}{\lp } + \norm{\agrad_{\disc_{N_0}} u_{N_0}-\agrad\overline{u}_0}{\lpd} \le 1.
 \]
We then build the sequence $(N_m)_{m\in\N}$ by induction.
 For any $m\in\N$, let $N_m>N_{m-1}$ such that there is $u_{N_m}\in X_{D_{N_m}}$ satisfying
 \[
  \norm{\api_{\disc_{N_m}} u_{N_m} - \overline{u}_m}{\lp } + \norm{\agrad_{\disc_{N_m}} u_{N_m}-\agrad\overline{u}_m}{\lpd} \le \frac 1 {m+1}.
 \]
 Then the sequence $(\norm{u_{N_m}}{\disc_{N_m}})_\mnn$ is bounded. Using the compactness hypothesis of  $(\disc_m)_\mnn$, there exists a subsequence, denoted
 $(\disc_{N_{\phi(m)}},u_{N_{\phi(m)}})_{\mnn}$ and $\overline{u}\in\lp$ such that $\api_{\disc_{N_{\phi(m)}}} u_{N_{\phi(m)}}$ converges to  $\overline{u}$ in $\lp$. The inequality
 \[
   \norm{\overline{u} - \overline{u}_{\phi(m)}}{\lp }  \le  \frac 1 {\phi(m)+1} + \norm{\api_{\disc_{N_{\phi(m)}}} u_{N_{\phi(m)}} - \overline{u}}{\lp }
 \]
 then shows that the subsequence $(\overline{u}_{\phi(m)})_{\mnn}$ converges to $\overline{u}$ in $\lp$.
\end{proof}

\subsubsection{A generic example of gradient discretisations}

In \cite{koala}, one can find a series of examples of nonconforming GDs, in the setting of the introduction of this paper (for usual 2nd order elliptic problems):
\begin{enumerate}
 \item Non-conforming finite elements,
 \item Discontinuous Galerkin methods,
 \item Hybrid Mimetic and Mixed methods.
\end{enumerate}
Definition \ref{def:galerkingd} below  gives a very simple example of a conforming GD, whose interest is to yield, in this abstract setting, the existence of at least one family of GDs that satisfies all required properties. 
Note that this particular gradient discretisation is the classical Galerkin approximation. 

\begin{definition}[Galerkin gradient discretisation]\label{def:galerkingd}
Let $(u_i)_{i\in\N}$ be a dense sequence in $\wunp$ (whose existence is ensured by the separability of $\wunp$).
For all $m\in\N$, define a conforming Galerkin gradient discretisation $\disc_m= (X_{\disc_m},\api_{\disc_m},\agrad_{\disc_m})$, in the sense of Definition \ref{def:graddisc}, in the following way:
\begin{enumerate}
 \item $X_{\disc_m}$ is the vector space spanned by $(u_i)_{i=0,\ldots,m}$,
 \item for all $u\in X_{\disc_m}$, $\api_{\disc_m}u = u$,
\item for all $u\in X_{\disc_m}$, $\agrad_{\disc_m}u = \agrad u$. 
\end{enumerate}
\end{definition}
\begin{lemma}[Existence of a coercive, consistent and limit-conforming (and compact) sequence of GDs]\label{lem:existgd}
~ The sequence $(\disc_m)_\mnn$ defined by Definition \ref{def:galerkingd} is coercive, limit-conforming and consistent in the sense of the definitions \ref{abs:def-coer}, \ref{abs:def-limconf} and \ref{abs:def-cons}. 
If we moreover assume that the embedding of $\wunp$ in $\lp$ is compact, then  $(\disc_m)_{\mnn}$ is also compact in the sense of Definition \ref{abs:def-comp}.
\end{lemma}

\begin{proof} By definition, for all $v \in X_{\disc_m}$, we have $\norm{v}{\disc_m} = \norm{v}{\wunp}$, which proves that $\norm{\cdot}{\disc_m}$ is a norm on $X_{\disc_m}$. The coercivity is then a consequence of Assumption \eqref{eq:hypV}. Relation \eqref{abs:stokes-formula} implies that $W_\disc$ defined by \eqref{abs:defwdisc} is identically null, which implies the limit-conformity property. The consistency is a consequence of the assumption that $(u_i)_{i\in\N}$ is a dense sequence in $\wunp$.
The compactness of the sequence is a straightforward consequence of the compact embedding  of $\wunp$ in $\lp$.
\end{proof}

\section{Approximation of an abstract Leray-Lions problem}\label{sec:absleraylions}

In this section, we generalise the problem presented in the introduction of this paper and provide
a convergence analysis of it based on the GDM.

Our general assumptions are similar to the assumptions considered in \cite{ler-65-res}:
\begin{subequations}
\begin{align}
&\begin{array}{l}
\mbox{$\bfa:\lp\times\lpd\to\lpd'$ is such that $\bfa(\cdot,\bv)$ is continuous for the strong topology of $\lpd'$,}\\
\mbox{$\bfa(v,\cdot)$ is continuous for the weak-$\star$ topology of $\lpd'$, and there exists $\overline{\alpha}\ge 0$ satisfying}\\
\forall v\in\lp,\ \forall \bv \in\lpd,  \norm{\bfa(v,\bv)}{\lpd'} \le \overline{\alpha}(1 + \norm{v}{\lp}^{p-1}+\norm{\bv}{\lpd}^{p-1}),
\end{array}
  \label{eq:defll}\\
&\begin{array}{l}
\mbox{$\bfa$ is monotonous in the sense:} \quad
\forall v\in\lp,\ \forall \bv,\bw \in\lpd,  \langle\bfa(v,\bv) - \bfa(v,\bw),\bv - \bw\rangle_{\lpd',\lpd} \ge 0,
\end{array}
 \label{eq:defmonot}\\
& 
\begin{array}{l}
\mbox{$\bfa$ is coercive in the sense: there exist $\underline{\alpha}>0$ such that}\\
\forall v\in\lp,\ \forall \bv\in\lpd, \ \underline{\alpha} \norm{\bv}{\lpd}^p \le \langle\bfa(v,\bv),\bv \rangle_{\lpd',\lpd}.
\end{array}
 \label{eq:defcoer}\\
&\begin{array}{l}
\mbox{$a:\lp\to \lppzero$ is continuous  for the weak-$\star$ topology of $\lp'$ and }
\forall v\in\lp,\   \norm{a(v)}{\lp'} \le \overline{\alpha}(1 + \norm{v}{\lp}^{p-1}),
\end{array}
  \label{eq:deflla}\\
&\begin{array}{l}
\mbox{$a$ is monotonous in the sense:} \quad
\forall v,w\in\lp,  \langle a(v) - a(w),v - w\rangle_{\lp',\lp} \ge 0,
\end{array}
 \label{eq:defmonota}\\
& 
\begin{array}{l}
\mbox{$a$ is ``$\lppzero$-coercive'' in the sense:}\quad
\forall v\in\lp, \ \underline{\alpha} \snorm{v}{\lp,\lppzero}^p \le \langle a(v),v \rangle_{\lp',\lp}.
\end{array}
 \label{eq:defcoera}
 \end{align}
\label{hypgnl}
\end{subequations}

The next two results ensure that for any separable reflexive smooth Banach spaces, there exist  operators $\bfa$ and $a$ with the required properties.
Let us recall the definition of a smooth Banach space.

\begin{definition}[Strictly convex and smooth Banach spaces]
\label{def:lindenstrauss} 
~

 	\begin{enumerate}
	  	\item A Banach space $(B,\norm{\cdot}{B})$ is said to be strictly convex if $\norm{\cdot}{B}$ is a strictly convex mapping from $B$ to $\R$,
		\item A Banach space $(B,\norm{\cdot}{B})$ is said to be smooth if, for any $x\in B$ with $\norm{x}{B} = 1$, there exists one and only one $f\in B'$ such that $f(x) = \norm{f}{B'}=1$,
  		\item If $(B,\norm{\cdot}{B})$ is smooth (resp. strictly convex), then $(B',\norm{\cdot}{B'})$ is strictly convex (resp. smooth).
 	\end{enumerate}
 \end{definition}

\begin{remark}[Equivalent strictly convex and smooth norms]
\label{rem:lindenstrauss} 
	Lindenstrauss proved in \cite{lindenstrauss} that any reflexive Banach space has an equivalent strictly convex and an equivalent smooth norm.
\end{remark}

\begin{lemma}[Existence of $\bfa$]\label{lem:plapgen}
 Assume that $\lpd$ is smooth 
 and define the duality mapping
 $\bfa:\lpd\to\lpd'$ by: for any $\bv\in \lpd$, $\bfa(\bv)\in\lpd'$ is the unique element such that
\be
\forall \bv\in\lpd,\ \langle\bfa(\bv),\bv\rangle_{\lpd',\lpd} = \norm{\bv}{\lpd}^p \hbox{ and }\norm{\bfa(\bv)}{\lpd'} =  \norm{\bv}{\lpd}^{p-1}.
\label{eq:defpsi}\ee
Then $\bfa$ satisfies Assumptions \eqref{eq:defll}--\eqref{eq:defcoer}.
\end{lemma}

\begin{proof}

From \cite{beurling,browder1965,Browder2013}, the mapping $\bfa$ exists (it is the so-called ``duality mapping'' associated to the constant gauge function equal to $1$) and is continuous for the weak-$\star$ topology of $\lpd'$.
The existence of $\bfa(\bv)$ such that \eqref{eq:defpsi} holds is provided by the Hahn-Banach theorem, and the uniqueness is a consequence of the fact that the norm of $\lpd'$ is strictly convex. 

The boundedness mentioned in \eqref{eq:defll} is obvious (with $\overline{\alpha}=1$), as well as the coercivity \eqref{eq:defcoer} (with $\underline{\alpha}=1$). It remains to check the monotonicity of $\bfa$. 
By developing the duality product and using the definition of $\bfa$,
 \[
  \langle\bfa(\bv) - \bfa(\bw),\bv - \bw\rangle_{\lpd',\lpd} = \norm{\bv}{\lpd}^p + \norm{\bw}{\lpd}^p - \langle\bfa(\bv),\bw\rangle_{\lpd',\lpd} - \langle\bfa(\bw),\bv\rangle_{\lpd',\lpd}.
\]
Therefore
 \begin{align*}
  \langle\bfa(\bv) - \bfa(\bw),\bv - \bw\rangle_{\lpd',\lpd} \ge{}& \norm{\bv}{\lpd}^p + \norm{\bw}{\lpd}^p - \norm{\bv}{\lpd}^{p-1}\norm{\bw}{\lpd}- \norm{\bw}{\lpd}^{p-1}\norm{\bv}{\lpd}\\
  ={}& ( \norm{\bv}{\lpd}^{p-1} - \norm{\bw}{\lpd}^{p-1})(\norm{\bv}{\lpd}-\norm{\bw}{\lpd}) \ge 0,
 \end{align*}
since the function $s\to s^{p-1}$ is increasing on $\R^+$. 
\end{proof}

\begin{remark}
 In the case $\bv\in\lpd = L^p(\O)^d$, the operator $\bfa$ defined by \eqref{eq:defpsi}
is $\bfa(\bv) = |\bv|^{p-2}\bv$.
\end{remark}

\begin{lemma}[Existence of $a$]\label{lem:agen}
 Assume that $\lp$ is smooth. Define 
 \[
  \widetilde a(u):= {\rm argmax}\{ \langle \mu,u\rangle_{\lp',\lp}; \mu\in\lppzero\hbox{ such that }\norm{\mu}{\lp'} =  1\}.
 \]
Then $\langle\widetilde{a}(u),u\rangle_{\lp',\lp}=\snorm{u}{\lp,\lppzero}$ and
$a(u) = |\langle \widetilde a(u),u\rangle_{\lp',\lp}|^{p-1} \widetilde a(u)$ satisfies Hypotheses \eqref{eq:deflla}--\eqref{eq:defcoera}.
\end{lemma}
\begin{proof}
The relation $\langle\widetilde{a}(u),u\rangle_{\lp',\lp}=\snorm{u}{\lp,\lppzero}$ is an immediate consequence of the definition of $\widetilde{a}$ and $\snorm{\cdot}{\lp,\lppzero}$.
The proof that $a$ satisfies the required properties is similar to that of Lemma \ref{lem:plapgen}.
\end{proof}

\begin{remark}
If $V={\rm span}(\mu_1,\ldots,\mu_r)$, a possible choice of $a$ that satisfies 
\eqref{eq:deflla}--\eqref{eq:defcoera} is
\[
a(u)=\sum_{i=1}^r |\langle \mu_i,u\rangle_{\lp',\lp}|^{p-2}\langle \mu_i,u\rangle_{\lp',\lp}\ \mu_i.
\]
In the case $r=1$, this operator $a$ is the one defined in Lemma \ref{lem:agen}.
\end{remark}

For any $b\in (\wunp)'$ (the space of linear continuous forms for $\norm{\cdot}{\wunp,\mathcal G}$), the abstract Leray-Lions problem reads in its weak form
\be
\mbox{Find $\ubarre\in\wunp$  such that }
\forall v\in\wunp,\  \langle\bfa(\ubarre,\agrad \ubarre),\agrad v \rangle_{\lpd',\lpd} + \langle a(\ubarre),v \rangle_{\lp',\lp}   = \langle b,v \rangle_{(\wunp)',\wunp}.
\label{eq:weakwunpp}\ee

The following lemma will enable us to write a strong form for this problem.

\begin{lemma} If $b\in (\wunp)'$ then there exists $(f,\bF)\in \lp'\times\lpd'$ such that
 \[
  \forall v\in\wunp,\ \langle b,v\rangle_{(\wunp)',\wunp} =  \langle f,v \rangle_{\lp',\lp} -\langle \bF,\agrad v \rangle_{\lpd',\lpd}.
 \]
\end{lemma}
\begin{proof}
Let $I:\wunp\to \lp\times\lpd$ be the embedding $I(v)=(v,\agrad v)$.
Define $\widetilde{b}:{\rm Im}(I)\to \R$ by $\widetilde{b}(I(v))=\langle b,v\rangle_{(\wunp)',\wunp}$.
Then $\widetilde{b}$ is linear and
$|\widetilde b(I(v))| \le \norm{b}{(\wunp)'}\norm{v}{\wunp,\mathcal G}=\norm{b}{(\wunp)'} (\norm{v}{\lp} + \norm{\agrad v}{\lpd})$.
The Hahn-Banach extension theorem then enables us to extend $\widetilde{b}$ as a continuous
linear form on $\lp\times \lpd$. Any such form can be represented as
$\widetilde{b}(v,\bv)=\langle f,v\rangle_{\lp',\lp}-\langle\bF,\bv\rangle_{\lpd',\lpd}$ for
some $(f,\bF)\in \lp'\times\lpd'$, and the proof is complete by choice of $\widetilde{b}$ on ${\rm Im}(I)$.\end{proof}

Using $(f,\bF)$ provided by the preceding lemma, without loss of generality the problem \eqref{eq:weakwunpp} can be re-written as
\be
\begin{aligned}
&\mbox{Find $\ubarre\in\wunp$  such that, }\forall v\in\wunp,\\
&\langle\bfa(\ubarre,\agrad \ubarre),\agrad v \rangle_{\lpd',\lpd} + \langle a(\ubarre),v \rangle_{\lp',\lp}  = \langle f,v \rangle_{\lp',\lp} -\langle \bF,\agrad v \rangle_{\lpd',\lpd}.
\end{aligned}
\label{eq:weak}\ee
The strong form of Problem \eqref{eq:weak} reads: 
\be
\begin{aligned}
&\mbox{Find $\ubarre\in\wunp$  such that $\bfa(\ubarre,\agrad \ubarre) +\bF\in\hdiv$ and }
- \adiv \big(\bfa(\ubarre,\agrad \ubarre)+\bF\big) + a(\ubarre) = f.
 \label{eq:pbstrong} 
\end{aligned}
\ee
\begin{lemma} \label{lem:equivstrongweak}
Problems \eqref{eq:pbstrong}  and \eqref{eq:weak} are equivalent.
\end{lemma}
\begin{proof} 
 Let $\ubarre\in\wunp$ be a solution to Problem \eqref{eq:pbstrong}. 
The equation in this formulation is a relation between elements of $\lp'$. 
Applying this equation to a generic $v\in\wunp$ and using \eqref{abs:stokes-formula} shows that $\ubarre$ is a solution to Problem \eqref{eq:weak}. 

Reciprocally, take $\ubarre\in\wunp$ a solution to Problem \eqref{eq:weak}. Then
the equation in \eqref{eq:weak} shows that, for all $v\in\wunp$,
\[
\langle\bfa(\ubarre,\agrad \ubarre)+\bF,\agrad v \rangle_{\lpd',\lpd} +  \langle a(\ubarre)-f,v \rangle_{\lp',\lp}.
\]
By definition \eqref{abs:defhdiv} of $\hdiv$, this shows that $\bfa(\ubarre,\agrad \ubarre)+\bF\in \hdiv$ and that
$\adiv(\bfa(\ubarre,\agrad \ubarre)+\bF) = a(\ubarre)-f$,
which is exactly \eqref{eq:pbstrong}. 
\end{proof}


\begin{remark}[Existence of a solution to \eqref{eq:weak}]
\label{rem:exist.sol} 
The fact that, under the framework of this section, there exists at least one solution to Problem \eqref{eq:weak}, is a by-product of the convergence theorem \ref{thm:convgradsch} below and of the existence result Lemma \ref{lem:existgd}. But it is also as  a consequence of \cite[Th\'eor\`eme 1]{ler-65-res}, in which the Banach space denoted by $V$, corresponds to $\wunp$ in this paper, and in which the operators denoted by $\bm{A}(u)$ and $A(u,v)$ are defined by the following.
\begin{itemize}
 \item If we assume that $\bfa$ only depends on its second argument, we define $\bm{A}~:~\wunp\to \wunp'$, by:
 \[
 \forall u,w\in \wunp,\ \langle \bm{A}(u),w\rangle_{\wunp,\wunp'} = \langle\bfa(\agrad u),\agrad w \rangle_{\lpd',\lpd} + \langle a(u),w \rangle_{\lp',\lp}.
 \]
 Then, owing to the monotony property of $\bfa$ and $a$, \cite[Hypoth\`ese I]{ler-65-res} is satisfied.
 \item In the case where $\bfa$ may also depend on its first argument, if we moreover assume that  the embedding of $\wunp$ in $\lp$ is compact, we define $A~:~\wunp\times\wunp\to \wunp'$, by:
 \[
 \forall u,v,w\in \wunp,\ \langle A(u,v),w\rangle_{\wunp,\wunp'} = \langle\bfa(u,\agrad v),\agrad w \rangle_{\lpd',\lpd} + \langle a(v),w \rangle_{\lp',\lp}.
 \]
 Then, owing to  Assumptions \ref{hypgnl}, \cite[Hypoth\`ese II]{ler-65-res} is satisfied.
\end{itemize}
This justifies the fact that we call Problem  \eqref{eq:weak} an abstract Leray-Lions problem.
\end{remark}

Given a gradient discretisation $\disc$, the gradient scheme (GS) for Problem \eqref{eq:weak} is: find $u\in X_\disc$ such that
\be
\forall v\in X_\disc,\  \langle\bfa(\api_\disc u, \agrad_\disc u),\agrad_\disc v \rangle_{\lpd',\lpd} + \langle a(\api_\disc u ),\api_\disc v \rangle_{\lp',\lp}   = \langle f,\api_\disc v \rangle_{\lp',\lp} -\langle \bF,\agrad_\disc v \rangle_{\lpd',\lpd}.
\label{gradsch_gen}\ee

\begin{theorem}[Convergence of the GS, abstract Leray--Lions pro\-blems]\label{thm:convgradsch}
Under Assumptions \eqref{hypgnl}, take a sequence $(\disc_m)_{m\in\N}$ of GDs in the sense of Definition \ref{def:graddisc}, which is consistent, limit-conforming and compact in the sense of Definitions \ref{abs:def-cons}, \ref{abs:def-limconf} and \ref{abs:def-comp}.

Then, for any  $m\in\N$, there exists at least one $u_m\in X_{\disc_m}$ solution to  the gradient scheme \eqref{gradsch_gen}. Moreover:
\begin{itemize}
 \item If we assume that $\bfa$ only depends on its second argument, then there exists $\ubarre$ solution of \eqref{eq:weak} such that, up to a subsequence, $\api_{\disc_m}  u_m$ converges weakly in $\lp$ to $\ubarre$ and  $\agrad_{\disc_m} u_m$ converges weakly in $\lpd$ to $\agrad \ubarre$ as $m\to\infty$.
 \item In the case where $\bfa$ may also depend on its first argument, if we moreover assume that the sequence $(\disc_m)_{m\in\N}$ of GDs is compact in the sense of Definition \ref{abs:def-comp} (this assumption implies that the embedding of $\wunp$ in $\lp$ is compact, see Lemma \ref{lem:compgd}), then there exists $\ubarre$ solution of \eqref{eq:weak} such that, up to a subsequence, $\api_{\disc_m}  u_m$ converges strongly in $\lp$ to $\ubarre$ and  $\agrad_{\disc_m} u_m$ converges weakly in $\lpd$ to $\agrad \ubarre$ as $m\to\infty$.
\end{itemize}

In the case where the solution $\ubarre$ of \eqref{eq:weak} is unique, then the above convergence results hold for the whole sequence.
\end{theorem}

\begin{proof}

{\bf Step 1}: existence of a solution to the scheme.

\medskip

Let $\disc$ be a GD in the sense of Definition \ref{def:graddisc}.
We endow the finite dimensional space $X_{\disc}$ with an inner product $\langle~,~\rangle$
and we denote by $|\cdot|$ its related norm.
Define $F:X_{\disc}\to X_{\disc}$ as the function such that, if $u\in X_{\disc}$,
$F(u)$ is the unique element in $X_{\disc}$ which satisfies
\[
\forall v\in X_{\disc}\,,\quad \langle F(u),v\rangle = \langle\bfa(\api_\disc u, \agrad_\disc u),\agrad_\disc v \rangle_{\lpd',\lpd}  + \langle a(\api_\disc u),\api_\disc v \rangle_{\lp',\lp}.
\]
Likewise, we denote by $w\in X_{\disc}$ the unique element such that
\[
\forall v\in X_{\disc}\,,\quad \langle w,v\rangle = \langle f,\api_\disc v \rangle_{\lp',\lp}-\langle \bF,\agrad_\disc v \rangle_{\lpd',\lpd}.
\]
The assumptions on $\bfa$ and $a$ show that $F$ is continuous and that, for all $u\in X_{\disc}$,
$\langle F(u),u\rangle \ge \underline{\alpha}(\norm{\agrad_\disc u}{\lpd}^p+ \snorm{\api_\disc u}{\lp,\lppzero}^{p}) = \underline{\alpha}\norm{u}{\disc}^p$.
By equivalence of the norms on the final dimensional space $X_\disc$, this shows that $\langle F(u),u\rangle \ge \ctel{cst-equiv}|u|^p$
with $\cter{cst-equiv}$ not depending on $u$.
Hence $\lim_{|u|\to\infty} \frac{\langle F(u),u\rangle}{|u|}=+\infty$
and $F$ is surjective (see \cite{ler-65-res} or \cite[Theorem 3.3, page 19]{deimling}). 
There is therefore $u\in X_\disc$ such that $F(u)=w$, which means that $u$ is a solution
to \eqref{gradsch_gen}.

\medskip

{\bf Step 2}: convergence to a solution of the continuous problem.

\medskip

As in the statement of the theorem, assume that $u_m$ is a solution to \eqref{gradsch_gen} with
$\disc=\disc_m$.
Letting $v=u_m$ in \eqref{gradsch_gen} with $\disc=\disc_m$ and using \eqref{abs:defcoercivity}, \eqref{eq:defcoer} and \eqref{eq:defcoera}, we get
\[
\underline{\alpha} ( \norm{\agrad_{\disc_m} u_m}{\lpd}^{p} + \snorm{\api_{\disc_m} u_m}{\lp,\lppzero}^{p}) = \underline{\alpha}\norm{u_m}{\disc_m}^p\le (C_{\disc_m}  \Vert f\Vert_{\lp' } + \norm{\bF}{\lpd'})\norm{u_m}{\disc_m}.
\]
Thanks to the coercivity of the sequence of  GDs, this provides
an estimate on $\agrad_{\disc_m} u_m$ in $\lpd$ and on $\api_{\disc_m} u_m$ in $\lp$.
Lemma \ref{abs:lem:reghunzero} then gives $\ubarre\in \wunp$ such that, up to a
subsequence, $\api_{\disc_m}u_m\to \ubarre$ weakly in $\lp$ and
$\agrad_{\disc_m}u_m\to \agrad\ubarre$ weakly in $\lpd$.
In the case where $\bfa$ may depend on its first argument, by compactness of the sequence of GDs, we can
also assume that the convergence of $\api_{\disc_m}u_m$
to $\ubarre$ is strong in $\lp$. 

By Hypothesis \eqref{eq:defll}, the sequence $(\bfa(\api_{\disc_m} u_m,\agrad_{\disc_m} u_m))_{m\in\N}$ of elements of $\lpd'$ 
remains bounded in $\lpd'$
and converges therefore, up to a subsequence, to some $\bfA$ weakly in $\lpd'$, as $m\to\infty$. Similarly, by Hypothesis \eqref{eq:deflla}, the sequence $a(\api_{\disc_m} u_m)$ of elements of $\lp'$ 
remains bounded in $\lp'$
and converges therefore, up to a subsequence, to some $A$ weakly in $\lp'$, as $m\to\infty$.

Let us now show that $\ubarre$ is solution to \eqref{eq:weak}, using the well-known Minty trick \cite{min-63-mon}. 
For a given $\varphi\in \wunp$ and for any gradient discretisation  $\disc$ in the sequence $(\disc_m)_{m\in\N}$, we introduce
\[
I_{\disc} \varphi \in \argmin_{v\in X_{\disc}}\left(\Vert \api_\disc v - \varphi\Vert_{\lp} + \Vert \agrad_\disc v -
\agrad\varphi\Vert_{\lpd}\right)
\]
as a test function in \eqref{gradsch_gen}. By the consistency of $(\disc_m)_{m\in\N}$, 
$\api_{\disc_m}I_{\disc_m}\varphi\to \varphi$ in $\lp$ and $\agrad_{\disc_m}I_{\disc_m}\varphi\to
\agrad\varphi$ in $\lpd$, as $m\to\infty$. Hence, letting $m\to\infty$ in the gradient scheme,
\be
\langle\bfA,\agrad \varphi\rangle_{\lpd',\lpd} + \langle A,\varphi \rangle_{\lp',\lp}   = \langle f,\varphi \rangle_{\lp',\lp}-\langle \bF,\agrad\varphi \rangle_{\lpd',\lpd},
 \ \ \forall \varphi \in \wunp.
\label{lltestphi}\ee
On the other hand, we may let $m\to\infty$ in  \eqref{gradsch_gen} with $u_m$ as a test function.
Using \eqref{lltestphi} with $\varphi=\ubarre$, this leads to
\begin{multline}
\lim_{m\to\infty} \big( \langle\bfa(\api_{\disc_m} u_m, \agrad_{\disc_m} u_m),\agrad_{\disc_m} u_m\rangle_{\lpd',\lpd} + \langle a(\api_{\disc_m} u_m),\api_{\disc_m} u_m \rangle_{\lp',\lp} \big) \\
= \langle f,\ubarre\rangle_{\lp',\lp} -\langle \bF,\agrad\ubarre \rangle_{\lpd',\lpd} = 
\langle\bfA,\agrad \ubarre\rangle_{\lpd',\lpd} +  \langle a(\ubarre),\ubarre \rangle_{\lp',\lp}.
\label{astucemonotdeux}\end{multline}
Hypotheses \eqref{eq:defmonot} and \eqref{eq:defmonota} give, for any $\vbarre \in\wunp$,
\begin{multline}
\langle\bfa(\api_{\disc_m} u_m, \agrad_{\disc_m} u_m) - \bfa(\api_{\disc_m} u_m,\agrad \vbarre),\agrad_{\disc_m} u_m - \agrad\vbarre\rangle_{\lpd',\lpd}\\
+ \langle a(\api_{\disc_m} u_m) - a(\vbarre), \api_{\disc_m} u_m - \vbarre \rangle_{\lp',\lp}\ge 0.
\label{eq:defam}\end{multline}
Developing this, using \eqref{astucemonotdeux} to identify the limit of
the sole term 
$\langle\bfa(\api_{\disc_m} u_m, \agrad_{\disc_m} u_m),\agrad_{\disc_m} u_m\rangle_{\lpd',\lpd}
+\langle a(\api_{\disc_m} u_m),\api_{\disc_m} u_m \rangle_{\lp',\lp}$
involving a product
of two weak convergences and using the (strong) continuity of $\bfa$ with respect to its first argument (the second argument is $\agrad\vbarre$), we may let $m\to\infty$ to get
\[
\langle\bfA - \bfa(\ubarre,\agrad \vbarre),\agrad\ubarre - \agrad\vbarre\rangle_{\lpd',\lpd} + \langle A-a(\vbarre),\ubarre - \vbarre \rangle_{\lp',\lp}\ge 0.
\]
Set $\vbarre = \ubarre + s v$ in the preceding inequality, where $v\in\wunp$ and $s >0$. Dividing by $s$, we get
\[
\langle\bfA - \bfa(\ubarre,\agrad \ubarre + s\agrad v),\agrad v\rangle_{\lpd',\lpd} + \langle A-a(\ubarre + s v),v\rangle_{\lp',\lp}\ge 0.
\]
Letting $s\to 0$ and using the continuity of $\bfa(\ubarre,\cdot)$ for the weak topology of $\lpd'$ and the continuity of $a$ for the weak topology of $\lp'$ leads to
\[
\langle\bfA - \bfa(\ubarre,\agrad \ubarre),\agrad v\rangle_{\lpd',\lpd} +  \langle A-a(\ubarre),v\rangle_{\lp',\lp}\ge 0,\ \forall v\in\wunp.
\]
Changing $v$ into $-v$ shows that $\langle\bfA,\agrad v\rangle_{\lpd',\lpd}+  \langle A,v\rangle_{\lp',\lp}=\langle\bfa(\ubarre,\agrad \ubarre),\agrad v\rangle_{\lpd',\lpd}+  \langle a(\ubarre),v\rangle_{\lp',\lp}$. Using this relation in \eqref{lltestphi} with $\varphi =v$, this concludes the proof that  $\ubarre$ is a solution of \eqref{eq:weak}. 
\end{proof}

We now have the tools for the proof of Theorem \ref{abs:thm:normeq} which gives an necessary and condition for \eqref{eq:hypV} to hold.
\medskip

 {\bf Proof of Theorem \ref{abs:thm:normeq}}
Let us assume that \eqref{abs:suppl} holds.
Since $\norm{\cdot}{\wunp,\mathcal{G}}$ is a norm, proving its equivalence with $\norm{\cdot}{\wunp}$ establishes that this latter semi-norm is also a norm.

Half of the equivalence has already been established in \eqref{un.sens}.
To prove the other half, we just need to show that
\[
E=\left\{ u\in \wunp \,:\,\norm{u}{\wunp}=1\right\}
\] 
is bounded in $\lp $. Indeed, this establishes the existence of $M\ge 0$ such that, for
all $u\in E$, $\norm{u}{\lp,\lppzero}\le M$ and thus, since $\norm{\agrad u}{\lpd}\le \norm{u}{\wunp}= 1$,
\[
\norm{u}{\wunp,\mathcal{G}}\le (1+M^p)^{1/p}=(1+M^p)^{1/p}\norm{u}{\wunp}.
\]
By homogeneity of the semi-norms, this concludes the proof that $\norm{\cdot}{\wunp,\mathcal{G}}$ and
$\norm{\cdot}{\wunp}$ are equivalent on $\wunp$.

To prove that $E$ is bounded, take $f\in \lp'$ and apply \eqref{abs:suppl} to get $\bv_f\in\hdiv$ and $\mu_f\in \lppzero$ such that $f = \adiv\bv_f + \mu_f$. Then, for any $u\in E$, by definition of the semi-norm $\snorm{\cdot}{\lp,\lppzero}$
and since $\norm{\agrad u}{\lpd}\le 1$ and $\snorm{u}{\lp,\lppzero}\le 1$,
\begin{align*}
|\langle f,u\rangle_{\lp',\lp}| =  | \langle \adiv\bv_f,u\rangle_{\lp',\lp} + \langle \mu_f,u\rangle_{\lp',\lp}|={}& | - \langle \bv_f,\agrad u\rangle_{\lpd',\lpd} + \langle \mu_f,u\rangle_{\lp',\lp}|\\
\le{}&  \norm{\bv_f}{\lpd'} \norm{\agrad u}{\lpd} +  \norm{\mu_f}{\lp'} \snorm{u}{\lp,\lppzero}\le  \norm{\bv_f}{\lpd'} +  \norm{\mu_f}{\lp'}.
\end{align*}
This shows that
$\{\langle f,u\rangle_{\lp',\lp}\,:\,u\in E\}$ is bounded by some constant depending on $f$. Since this is valid for any $f\in \lp'$,
the Banach--Steinhaus theorem \cite[Theorem 2.2]{brezis} shows that $E$ is bounded in $\lp$.

\medskip

Reciprocally, let us assume that $ \norm{\cdot}{\wunp} $ and $\norm{\cdot}{\wunp,\mathcal{G}} $  are two equivalent norms, and let us prove Property \eqref{abs:suppl}. 
Thanks to \cite{lindenstrauss} (see Remark \ref{rem:lindenstrauss}), we can assume that $(\lp,\norm{\cdot}{\lp})$ and $(\lpd,\norm{\cdot}{\lpd})$ are smooth. 
Lemma \ref{lem:plapgen} can then be applied to define  $\bfa$ by \eqref{eq:defpsi}.
Let $a:\lp\to \lppzero\subset \lp'$ be defined as in Lemma \ref{lem:agen}.
Thanks to the remark\ref{rem:exist.sol}, for any $f\in \lp'$, there exists a solution $\bu$ to \eqref{eq:pbstrong} with $\bF=0$. 
Setting $\bv=-\bfa(\agrad \bu)$, we see that $f=\adiv\bv + a(\bu)\in {\rm Im}(\adiv)+\lppzero$ and the proof is complete.
 \eop


\section{Approximation of a linear elliptic problem}\label{sec:linell}

We consider here a particular case of Problem \eqref{eq:pbstrong}/\eqref{eq:weak}.
We take $p=2$ and make the following assumptions:
\begin{subequations}
\begin{align}
&
\begin{array}{l}
\mbox{$\bfa:\lpd\to\lpd'$ is linear continuous with norm bounded by $\overline{\alpha}$,}
\end{array}
  \label{eq:lindefll}\\
& 
\begin{array}{l}
\mbox{$\bfa$ is coercive in the sense: there exists $\underline{\alpha}>0$ such that}\\
\forall v\in\lp,\ \forall \bv\in\lpd, \ \underline{\alpha} \norm{\bv}{\lpd}^2 \le \langle\bfa(\bv),\bv \rangle_{\lpd',\lpd}.
\end{array}
 \label{eq:lindefcoer}\\
&\begin{array}{l}
\mbox{$a:\lp\to \lp'$ is linear and continuous  with norm bounded by $\overline{\alpha}$ },
\end{array}
  \label{eq:lindeflla}\\
& 
\begin{array}{l}
\mbox{$a$ is coercive in the sense:}\quad
\forall v\in\lp, \ \underline{\alpha} \snorm{v}{\lp,\lppzero}^2 \le \langle a(v),v \rangle_{\lp',\lp}.
\end{array}
 \label{eq:lindefcoera}
 \end{align}
\label{hypgl}
\end{subequations}

Then $\lpd$ is a Hilbert space when endowed with the scalar product 
\[
(\bv,\bw)\mapsto \half\left[\langle\bfa(\bv),\bw\rangle_{\lpd',\lpd}+\langle\bfa(\bw),\bv\rangle_{\lpd',\lpd}\right].
\]
Hypotheses \eqref{hypgl} imply
\be
\forall u\in\wunp,\ \underline{\alpha}\norm{u}{\wunp}^2 \le \langle\bfa(\agrad u),\agrad u\rangle_{\lpd',\lpd}+ \langle a(u),u \rangle_{\lp',\lp} \le \overline{\alpha}\norm{u}{\wunp}^2,
\label{eq:equivnormwunp}\ee
which shows that $\wunp$ is a Hilbert space when endowed with the scalar product 
\be
(u,v)\mapsto \langle u,v\rangle_{\wunp} : =\half\left[\langle\bfa(\agrad u),\agrad v\rangle_{\lpd',\lpd}+\langle\bfa(\agrad v),\agrad u\rangle_{\lpd',\lpd} + \langle a(u),v \rangle_{\lp',\lp}+\langle a(v),u \rangle_{\lp',\lp} \right].
\label{def:pscawunp}\ee

\begin{remark}If $\lp$ and $\lpd$ are Hilbert spaces, identifying $\lp$ with $\lp'$ and $\lpd$ with $\lpd'$, then $\snorm{u}{\lp,\lppzero} = \norm{P_\lppzero(u)}{\lp}$, denoting by $P_V$ the orthogonal projection on $\lppzero$. Then $\bfa$ and $a$ constructed
in Lemmas \ref{lem:plapgen} and \ref{lem:agen} with $p=2$ satisfy $\bfa = \rm{I}_d$ and $a = P_\lppzero$.
\end{remark}

For any  $(f,\bF)\in \lp'\times \lpd'$,  the abstract linear elliptic problem reads, in its weak form,
\be
\begin{aligned}
&\mbox{Find $\ubarre\in\wunp$  such that, }
\forall v\in\wunp,\  \langle\bfa(\agrad \ubarre),\agrad v \rangle_{\lpd',\lpd} + \langle a(\ubarre),v \rangle_{\lp',\lp}  = \langle f,v \rangle_{\lp',\lp} -\langle \bF,\agrad v \rangle_{\lpd',\lpd}
\end{aligned}
\label{eq:linweak}\ee
and, in its strong form,
\be
\begin{aligned}
&\mbox{Find $\ubarre\in\wunp$  such that $\bfa(\agrad \ubarre) +\bF\in\hdiv$ and }
- \adiv \big(\bfa(\agrad \ubarre)+\bF\big) + a(\ubarre) = f.
 \label{eq:linpbstrong} 
\end{aligned}
\ee
As proved by Lemma \ref{lem:equivstrongweak}, Problems \eqref{eq:linpbstrong}  and \eqref{eq:linweak} are equivalent.
The following theorem ensures that
\eqref{eq:linpbstrong} and \eqref{eq:linweak} have exactly one solution. 

\begin{theorem}[Existence  and uniqueness of a solution to \eqref{eq:linweak}]\label{th:exist.uniq.sol.lin}
Under Hypothesis \eqref{hypgl}, there exists one and only one solution to Problem \eqref{eq:linweak}.
\end{theorem}
\begin{proof}
This is an immediate consequence of Lax-Milgram theorem, on the Hilbert space $\wunp$ endowed with the inner product defined by \eqref{def:pscawunp}
\end{proof}

Table \ref{tab:lin} presents the links between this abstract linear elliptic setting and
the standard elliptic PDE, for all BCs proposed in the introduction of this paper.

\begin{table}[!h]
\begin{center}
\begin{tabular}{|c|c|c|c|c|}
 \hline
 B.C. & \begin{tabular}{c}homogeneous\\Dirichlet\end{tabular} & \begin{tabular}{c}homogeneous\\Neumann \end{tabular}& \begin{tabular}{c}non-homogeneous\\ Neumann \end{tabular} & Fourier \\
 \hline
 $\lp$ & $L^2(\O)$ &  $L^2(\O)$ & $L^2(\O)\times L^2(\dr\O)$ & $L^2(\O)\times L^2(\dr\O)$ \\
 \hline
 $\lpd$ & $L^2(\O)^d$ &  $L^2(\O)^d$ & $L^2(\O)^d$ & $L^2(\O)^d$ \\
 \hline
 $\bfa:$ & $\ba \bv\mapsto
 \Lambda \bv\ea$ &  $\ba \bv\mapsto
 \Lambda \bv\ea$ & $\ba \bv\mapsto
 \Lambda \bv\ea$  & $\ba \bv\mapsto
 \Lambda \bv\ea$ \\
 \hline
 $a:$ & $u\mapsto 0$ &  $\ba u\mapsto 
 \int_\O u 1_\Omega \ea$ & $\ba (u,w)\mapsto
 \int_\O u (1_\Omega,0)\ea$  & $\ba(u,w)\mapsto
 (0,b \tr u)\ea$ \\
 \hline
\end{tabular}
\caption{Link between the abstract linear elliptic problem and the usual
elliptic PDE $-\div(\Lambda\nabla\bu)=f+\div(\mathbf{F})$, for various various boundary conditions.}
\label{tab:lin}
\end{center}
\end{table}

Given a gradient discretisation $\disc$ in the sense of Definition \ref{def:graddisc},
we consider the following scheme for the approximation of Problem \eqref{eq:linweak}: find $u\in X_\disc$ such that
\be
\forall v\in X_\disc,\  \langle\bfa(\agrad_\disc u),\agrad_\disc v \rangle_{\lpd',\lpd} + \langle a(\api_\disc u ),\api_\disc v \rangle_{\lp',\lp}   = \langle f,\api_\disc v \rangle_{\lp',\lp} -\langle \bF,\agrad_\disc v \rangle_{\lpd',\lpd}.
\label{eq:lingradsch}\ee
Fixing a basis $(\basex^{(i)})_{i=1,\ldots,N}$ of $X_{\disc}$, the scheme \eqref{eq:lingradsch} is equivalent to solving the  linear square system $A U = B$, where
\begin{eqnarray}
&& u = \sum_{j=1}^N U_j \basex^{(j)},\nonumber\\
 && A_{ij} = \langle\bfa(\agrad_\disc \basex^{(j)}),\agrad_\disc \basex^{(i)} \rangle_{\lpd',\lpd} + \langle a(\api_\disc \basex^{(j)} ),\api_\disc \basex^{(i)} \rangle_{\lp',\lp},\label{eq:gradsch_lin_sl}\\
 && B_i = \langle f,\api_\disc \basex^{(i)}  \rangle_{\lp',\lp} -\langle \bF,\agrad_\disc \basex^{(i)}  \rangle_{\lpd',\lpd}.\nonumber
\end{eqnarray}

\begin{theorem}[Error estimates, abstract linear elliptic pro\-blem]\label{thm:convgradschlin}
Under Assumptions \eqref{hypgl}, let $\disc$ be a GD in the sense of Definition \ref{def:graddisc}.
Then there exists one and only one $u_\disc\in X_{\disc}$ solution to the GS \eqref{eq:lingradsch}. This solution satisfies the following inequalities:
\begin{align}
& \Vert \agrad \bu - \agrad_\disc u_\disc\Vert_{\lpd}
\le \frac 1 {\underline{\alpha}}\left[ W_\disc(\bfa(\agrad\bu)+\bm{F}) + (\overline{\alpha}(1+C_\disc) +\underline{\alpha})S_\disc(\bu)\right],
\label{errggradsch}\\
& \Vert \bu -\api_\disc  u_\disc\Vert_{\lp }
\le   \frac 1 {\underline{\alpha}} \left[C_\disc W_\disc(\bfa(\agrad\bu)+\bm{F}) + (C_\disc (1+C_\disc)\overline{\alpha} +\underline{\alpha}) S_\disc(\bu)\right],
\label{errugradsch}
\end{align}
where $C_\disc$, $S_\disc$ and $W_\disc$  are respectively the norm of the reconstruction operator $\api_\disc$, the consistency defect and the conformity defect, defined by  \eqref{abs:defcoercivity}, \eqref{abs:defsdisc} and \eqref{abs:defwdisc}.

Moreover, we also have the reverse inequalities
\begin{align}
&  W_\disc(\bfa(\agrad\bu)+\bm{F})\le \overline{\alpha}\Vert \agrad \bu - \agrad_\disc u_\disc\Vert_{\lpd},
\label{errggradschinv}\\
&  S_\disc(\bu) \le \Vert \bu -\api_\disc  u_\disc\Vert_{\lp } + \Vert \agrad \bu - \agrad_\disc u_\disc\Vert_{\lpd},
\label{errugradschinv}
\end{align}
which shows the existence of $\ctel{cte:eqinf}>0$ and $\ctel{cte:eqsup}>0$, only depending on $\overline{\alpha}$ and $\underline{\alpha}$, such that
\begin{multline}
\frac {\cter{cte:eqinf}}{1+C_\disc} \left[S_\disc(\bu) + W_\disc(\bfa(\agrad\bu)+\bm{F})\right]
\le \Vert \bu -\api_\disc  u_\disc\Vert_{\lp } +\Vert \agrad \bu - \agrad_\disc u_\disc\Vert_{\lpd} \\
\le \cter{cte:eqsup} (1+C_\disc)^2\left[S_\disc(\bu) + W_\disc(\bfa(\agrad\bu)+\bm{F})\right].
\label{equiverror}
\end{multline}
\end{theorem}

\begin{proof}

Let us first prove that, if \eqref{errggradsch}--\eqref{errugradsch} holds for any solution $u_\disc\in X_{\disc}$ to Scheme \eqref{eq:lingradsch}, then the solution to this scheme exists and is unique.
For that, we prove that if  \eqref{errggradsch} holds then the matrix denoted by $A$ of the linear system \eqref{eq:gradsch_lin_sl} is non-singular. 
This will be completed if we prove $A U = 0$ implies $U=0$. Thus, we consider the particular case where $f=0$ and $\bm{F}=0$
which gives a zero right-hand side. In this case the solution $\bu$ of \eqref{eq:linweak} is equal to zero.
Then from \eqref{errggradsch}--\eqref{errugradsch}, any solution to the scheme satisfies $\norm{ u_\disc }{\disc}=0$.
Since $\norm{ \cdot }{\disc}$ is a norm on $X_{\disc}$ this leads to $u_\disc=0$.
Therefore \eqref{eq:gradsch_lin_sl} (as well as \eqref{eq:lingradsch}) has a unique solution for any right-hand side $f$ and~$\bm{F}$.

Let us now prove that any solution $u_\disc\in X_{\disc}$ to Scheme \eqref{eq:lingradsch} satisfies \eqref{errggradsch} and \eqref{errugradsch}.
Notice that $\bm{\varphi} = \bfa(\agrad\bu) +\bm{F} \in \hdiv$ and can thus be considered in the definition \eqref{abs:defwdisc} of $W_\disc$. This gives, for any $v\in X_{\disc}$,
\[
\left\vert\langle \bfa(\agrad\bu) +\bm{F},\agrad_\disc v\rangle_{\lpd',\lpd} + \langle\adiv(\bfa(\agrad\bu)+\bm{F}),\api_\disc v\rangle_{\lp',\lp} \right\vert \le\Vert  v\Vert_{\disc}\  W_\disc(\bfa(\agrad\bu)+\bm{F}).
\]
Since $-f + a(\ubarre) =  \adiv( \bfa(\agrad\bu) +\bm{F})$, this yields
\be
\left\vert \langle \bfa(\agrad\bu) +\bm{F},\agrad_\disc v\rangle_{\lpd',\lpd} + \langle-f + a(\ubarre),\api_\disc v\rangle_{\lp',\lp} \right\vert \le\Vert  v\Vert_{\disc}\ W_\disc(\bfa(\agrad\bu)+\bm{F}).
\label{for.non.homo.0}
\ee
Using the gradient scheme \eqref{eq:lingradsch} to replace the terms involving $f$ and $\mathbf{F}$ in the left-hand side, we infer
\be
\label{for.non.homo}
\left\vert\langle \bfa(\agrad\bu - \agrad_\disc u_\disc),\agrad_\disc v\rangle_{\lpd',\lpd} + \langle a(\ubarre-\api_\disc  u_\disc ),\api_\disc v\rangle_{\lp',\lp}\right\vert\le\Vert  v\Vert_{\disc}\ W_\disc(\bfa(\agrad\bu)+\bm{F}).
\ee
Define $\Idisc \bu = \argmin_{w\in X_{\disc}}(\Vert \api_\disc w -\bu\Vert_{\lp } + \Vert \agrad_\disc w - \agrad\bu\Vert_{\lpd})$
and notice that, by definition \eqref{abs:defsdisc} of $S_\disc$,
\be\label{for.non.homo2}
\norm{\api_\disc \Idisc\bu-\bu}{\lp }+\norm{\agrad_\disc \Idisc\bu-\agrad\bu}{\lpd}
= S_\disc(\bu).
\ee
Recalling the definition
of $\norm{\cdot}{\disc}$ in Definition \ref{def:graddisc}, introducing
$\agrad\bu$ and $\api\bu$ and using \eqref{for.non.homo}
gives
\begin{align*}
\langle \bfa(\agrad_\disc \Idisc \bu &- \agrad_\disc u_\disc),\agrad_\disc v\rangle_{\lpd',\lpd} + \langle a(\api_\disc \Idisc \bu-\api_\disc  u_\disc ),\api_\disc v\rangle_{\lp',\lp}\\
\le{}& \norm{ v}{\disc}\ W_\disc(\bfa(\agrad\bu)+\bm{F})
+ \left| \langle \bfa(\agrad_\disc \Idisc \bu - \agrad\bu),\agrad_\disc v\rangle_{\lpd',\lpd} + \langle a(\api_\disc \Idisc \bu-\bu ),\api_\disc v\rangle_{\lp',\lp}\right|\\
\le{}&  \norm{ v}{\disc}\left[ W_\disc(\bfa(\agrad\bu)+\bm{F}) + \overline{\alpha}(\norm{\agrad_\disc \Idisc \bu - \agrad  \bu}{\lpd}+C_\disc\norm{\api_\disc \Idisc \bu -  \bu}{\lp})\right] \\
\le{}& \norm{ v}{\disc}\left[ W_\disc(\bfa(\agrad\bu)+\bm{F}) + \overline{\alpha}(1+C_\disc)S_\disc( \bu)\right].\end{align*}
Choose $v = \Idisc \bu - u_\disc$ and apply Hypothesis \eqref{hypgl}:
\be\label{for.est.PiD}
\underline{\alpha} \norm{\Idisc \bu - u_\disc}{\disc}
\le  W_\disc(\bfa(\agrad\bu)+\bm{F}) + \overline{\alpha}(1+C_\disc)S_\disc( \bu).
\ee
Estimate \eqref{errggradsch} follows by using the triangle inequality:
\begin{align}
\Vert \agrad\bu-{}&\agrad_\disc u_\disc\Vert_{\lpd} \le \norm{\agrad\bu-\agrad_\disc \Idisc\bu}{\lpd}
+\norm{\agrad_\disc(\Idisc\bu-u_\disc)}{\lpd}\nonumber\\
\le{}& \norm{\agrad\bu-\agrad_\disc \Idisc\bu}{\lpd}
+\norm{\Idisc\bu-u_\disc}{\disc}
\le S_\disc(\bu)+\frac{1}{\underline{\alpha}}\left(W_\disc(\bfa(\agrad\bu)+\bm{F}) + \overline{\alpha}(1+C_\disc)S_\disc( \bu)\right).
\label{for.est.PiD2}
\end{align}
Using \eqref{abs:defcoercivity} and \eqref{for.est.PiD}, we get 
\be\label{for.est.PiD3}
\underline{\alpha} \norm{\api_\disc \Idisc  \bu -\api_\disc  u_\disc}{\lp }
\le  C_\disc( W_\disc(\bfa(\agrad\bu)+\bm{F}) + \overline{\alpha}(1+C_\disc)S_\disc(  \bu)),
\ee
which yields  \eqref{errugradsch} by invoking, as in \eqref{for.est.PiD2}, the triangle inequality and the estimate
$\Vert  \bu - \api_\disc \Idisc \bu \Vert_{\lpd}  \le S_\disc  (\bu)$.

\medskip

Let us now turn to the proof of \eqref{errggradschinv}. The gradient scheme \eqref{eq:lingradsch}
gives, for any $v\in X_{\disc}\setminus\{0\}$,
\begin{multline*} 
\langle f - a(\bu),\api_\disc v \rangle_{\lp',\lp} -\langle \bfa(\agrad\bu)+\bF,\agrad_\disc v \rangle_{\lpd',\lpd} = \langle\bfa(\agrad_\disc u - \agrad\bu),\agrad_\disc v \rangle_{\lpd',\lpd} + \langle a(\api_\disc u -\bu),\api_\disc v \rangle_{\lp',\lp}
\end{multline*}
and thus
\[
\frac{\left\vert \langle f - a(\bu),\api_\disc v \rangle_{\lp',\lp} -\langle \bfa(\agrad\bu)+\bF,\agrad_\disc v \rangle_{\lpd',\lpd}\right\vert}{\norm{ v}{\disc} } 
 \le \overline{\alpha}(\norm{\agrad_\disc u -\agrad\bu}{\lpd} + C_\disc\norm{\api_\disc u -\bu}{\lp}).
\]
Taking the supremum over $v$ on the left hand side yields  \eqref{errggradschinv} since \eqref{eq:linpbstrong} holds. Inequality \eqref{errugradschinv} is an immediate consequence of the definition of $S_\disc(\bu)$.
\end{proof}

 \medskip

\begin{remark}[On the compactness assumption]\label{rem:nocomp}
Note that, in the linear case, the compactness of the sequence of GDs is not required to obtain the  convergence. This compactness assumption is in general only needed for some non-linear problems.
\end{remark}

\begin{remark}[consistency and limit-conformity are necessary conditions]
We state here a kind of reciprocal property to the convergence property. Let us assume that,  under Hypothesis  \eqref{hypgl}, a sequence $(\disc_m)_{m\in\N}$ of GDs is such that,   for all $f\in \lp $ and $\bm{F}\in \lpd$ and for all $m\in\N$, there exists $u_m\in X_{\disc_m}$ which is solution to the gradient scheme \eqref{eq:lingradsch} and such that $\api_{\disc_m}  u_m$ (resp. $\agrad_{\disc_m}  u_m$) converges in $\lp $ to the solution $\ubarre$ of \eqref{eq:linweak} (resp.  in $\lpd$ to $\agrad\ubarre$). Then $(\disc_m)_{m\in\N}$ is consistent and limit-conforming in the sense of Definitions \ref{abs:def-cons} and \ref{abs:def-limconf}.

Indeed, for $\varphi\in \wunp $, let us consider $f =  a(\varphi)$ and $\mathbf{F}=-\bfa(\agrad \varphi)$ in \eqref{eq:linweak}. Since in this case, $\ubarre = \varphi$, the assumption that  $\api_{\disc_m}  u_m$ (resp. $\agrad_{\disc_m}  u_m$) converges in $\lp $ to the solution $\varphi$ of \eqref{eq:linweak} (resp.  converges in $\lpd$ to $\agrad\varphi$), inequality \eqref{errugradschinv} proves that $S_{\disc_m}(\varphi)$ tends to 0 as $m\to\infty$, and therefore the sequence  $(\disc_m)_{m\in\N}$ is consistent.

For $\bvarphi \in \hdiv$, let us set $f = \adiv\bvarphi$ and $\bm{F}= -\bvarphi$ in \eqref{eq:linweak}. In this case, the
solution $\ubarre$ is equal to $0$, since the right-hand side  of \eqref{eq:linweak}  vanishes for any $v\in \wunp $.
Then inequality \eqref{errggradschinv} implies
\[
 W_{\disc_m}(\bvarphi) \le \overline{\alpha}  \Vert \agrad_{\disc_m}  u_m\Vert_{\lpd} \to 0 \hbox{ as }m\to 0,
\]
hence concluding that the sequence $(\disc_m)_{m\in\N}$  is limit-conforming.

Note that, if we now assume that $\agrad_{\disc_m} u_m$ converges only weakly (instead of strongly) in $\lpd$ to $\agrad\ubarre$, the same conclusion holds. 
Indeed, the other hypotheses on $(\disc_m)_\mnn$ are sufficient to prove that  $\agrad_{\disc_m}  u_m$ actually converges strongly in $\lpd$ to $\agrad\ubarre$.
It suffices to observe that
\[
  \lim_{m\to\infty}  (\langle f,\api_{\disc_m} u_m \rangle_{\lp',\lp} -\langle \bF,\agrad_{\disc_m} u_m\rangle_{\lpd',\lpd}) = 
  \langle f,\ubarre\rangle_{\lp',\lp} -\langle \bF,\agrad\ubarre\rangle_{\lpd',\lpd}.
\]
Then we take $v= \ubarre$ in \eqref{eq:linweak} and $v = u_m$ in \eqref{eq:lingradsch}, this leads to
  \begin{multline*}
   \lim_{m\to\infty}   (\langle\bfa(\agrad_{\disc_m} u_m),\agrad_{\disc_m} u_m \rangle_{\lpd',\lpd} + \langle a(\api_{\disc_m} u_m ),\api_{\disc_m} u_m \rangle_{\lp',\lp}) \\= 
   \langle f,\ubarre\rangle_{\lp',\lp} -\langle \bF,\agrad\ubarre\rangle_{\lpd',\lpd} = 
   \langle\bfa(\agrad\ubarre),\agrad\ubarre \rangle_{\lpd',\lpd} + \langle a(\ubarre),\ubarre\rangle_{\lp',\lp}.
  \end{multline*}
  In addition to the assumed weak convergence property of $\agrad_{\disc_m}  u_m$, this proves
  \[
   \lim_{m\to\infty} \langle\bfa(\agrad_{\disc_m} u_m-\agrad\ubarre),\agrad_{\disc_m} u_m -\agrad\ubarre\rangle_{\lpd',\lpd}  = 0,
  \] 
  and the convergence of $\agrad_{\disc_m} u_m$ to $\agrad\ubarre$ in $\lpd$ follows from the coercivity of $\bfa$ assumed in \eqref{hypgl}.

 \label{rem:reciprocv}
\end{remark}


\section{Other applications of the unified discretisation setting}\label{sec:otherapp}

We briefly present here other PDE models that can be analysed using the unified setting presented
in this paper.

\subsection{A hybrid-dimensional problem}

We consider a simplified model for a Darcy flow in a convex domain $\Omega\subset\mathbb{R}^3$, in which a fracture $\Gamma$ splits the domain $\Omega$ into two subdomains, $\Omega_1$ and $\Omega_2$. This fracture is defined by $\Gamma= \Omega\cap P$, where $P$ is a plane.  We assume that $\bm{n}_{12}$ is the unit vector normal to $\Gamma$, oriented from $\Omega_1$ to $\Omega_2$. For this problem, the continuous model reads
\begin{eqnarray}
\label{modeleCont}
\left\{\begin{array}{r@{\,\,}c@{\,\,}lll}
-\div(\Lambda \nabla u) &=& r &\mbox{ in } \Omega_i, \; i=1,2,&\\
u &=& 0 &\mbox{ on } \partial \Omega,&\\
-\div_\Gamma(\Lambda_\Gamma\nabla_\Gamma u)  + (\Lambda \nabla u_{|\Omega_1} -  \Lambda \nabla u_{|\Omega_2})\cdot \bm{n}_{12} &=& r_\Gamma &\mbox{ on } \Gamma,
\end{array}\right.
\end{eqnarray}
where $\nabla_\Gamma$ (resp. $\div_\Gamma$) is the $2D$ gradient (resp. divergence) along $\Gamma$, $r\in L^2(\O)$, $r_\Gamma\in L^2(\Gamma)$.

Defining the space
$$
H = \{ v \in H^1_0(\Omega) \,|\, \gamma_\Gamma v \in H^1(\Gamma)\},
$$ 
the weak formulation of Problem \eqref{modeleCont} is given by: find $\bar u\in V$ such that 
\begin{eqnarray}
\label{VarForm}
\forall v\in H,\ 
\dsp\int_\Omega \Lambda  \nabla \bar u \cdot \nabla v \d\x 
+ \int_{\Gamma}  \Lambda_\Gamma \nabla_\Gamma \gamma_\Gamma \bar u \cdot \nabla_\Gamma \gamma_\Gamma v \dfrontiere
= 
\dsp\int_\Omega r v  \d\x +  \int_{\Gamma} r_\Gamma\gamma_\Gamma v\dfrontiere. 
\end{eqnarray} 
This weak formulation is then identical to \eqref{eq:linweak} by letting:
\begin{itemize}
 \item $\lp = L^2(\O)\times L^2(\Gamma)$, $\lpd = L^2(\O)^3\times L^2(\Gamma)^2$,
 \item $\wunp = \{ (v, \gamma_\Gamma v), v\in H\}$ and $\agrad (v, \gamma_\Gamma v) = (\nabla v, \nabla_\Gamma \gamma_\Gamma v)$,
 \item $V = \{0\}$, $\bfa(\bv, \bw) = (\Lambda \bv, \Lambda_\Gamma \bw)$, $f = (r,r_\Gamma)$, $\bF = 0$.
\end{itemize}
Then, in this very simple case of fracture, the abstract Gradient Discretisation Method defined here applied to this problem is identical to that of \cite{brenner2016gradient}. It is expected that the general case of fractured domain studied in \cite{brenner2016gradient} could enter into this framework as well; this however will not avoid the tricky proof of density results done in \cite{brenner2016gradient}.

\subsection{Linear elasticity in solid continuum mechanics}

Consider now the following spaces:
\begin{itemize}
 \item $\O\subset\R^3$,
 \item $\lp = L^2(\O)^3$, so that $\lp'=L^{2}(\O)^3=\lp$.
 \item $\lpd = \lp ^{3\times 3}$, so that $\lpd'=\lp ^{3\times 3}$.
 \item $\hdiv = H_{\rm div}(\O)^3$, and $\lppzero =\{0\}$.
 \item $\wunp = H^{1}_0(\O)^3$.
\end{itemize}

The operators $\agrad:H^{1}_0(\O)^3\to \lp ^{3\times 3}$ and $\adiv:H_{\rm div}(\O)^3\to \lp ^3$ are defined, for $u\in H^{1}_0(\O)^3$ (the ``displacement field'') by
\[
 (\agrad u)_{i,j} = \frac 1 2 (\partial_i u^{(j)} + \partial_j u^{(i)}), 
\]
and, for $\sigma\in  H_{\rm div}(\O)^3$ (the ``stress field'')
\[
 (\adiv \sigma)_i = \sum_{j=1}^3 \partial_j \sigma^{(i,j)}.
\]
Then, the construction in Section \ref{sec:linell} handles the case of the linear elasticity theory in solid continuum mechanics. 
Indeed, a strong formulation of the equilibrium of a solid under internal forces is Problem \eqref{eq:linpbstrong}, the linear operator $\bfa$ expresses Hooke's law ($\bfa(\agrad u)_{i,j} = \lambda \sum_{k=1}^3 (\agrad u)_{k,k}\delta_{i,j} + 2\mu (\agrad u)_{i,j} $ with $\delta_{i,j}=1$ if $i=j$ and $0$ otherwise, and $\lambda\ge 0, \mu>0$ are given) and \eqref{eq:linweak} is the so-called ``virtual displacement'' formulation, that is the weak formulation of  \eqref{eq:linpbstrong}.

\subsection{Riemannian geometry}

Let $(M,g)$ be a compact orientable Riemannian manifold of dimension $d$ without boundary, and corresponding measure $\mu_g$. 
We denote by $TM=\cup_{x\in M}(\{x\}\times T_xM)$ the tangent bundle to $M$, and define the operators and spaces
\begin{itemize}
 \item $\lp = L^2(M)$, so that $\lp'=L^{2}(M)=\lp$,
 \item $\lpd = L^2(TM):=\{\bv\,:\,\bv(x)\in T_xM\,,\forall x\in M\mbox{ and }x\mapsto g_x(\bv(x),\bv(x))^{1/2} \in L^2(M)\}$; we have $\lpd'=\lpd$,
 \item $\agrad:C^1(M)\to L^2(TM)$ the standard gradient, that is $\agrad u=\nabla_g u$ such that, for any smooth vector field $X$ and any $x\in M$, $\nabla_g u(x)\in T_xM$ and $g_x(X(x),\nabla_g u(x))=du_x(X(x))$. 
 \item $\wunp$ is the closure in $L^2(M)$ of $C^1(M)$ for the norm
\[
u\mapsto \left(\int_M |u(x)|^2\,d\mu_g(x) + \int_M g_x(\nabla_g u(x),\nabla_g u(x))\,d\mu_g(x) \right)^{1/2}.
\]
Then $\agrad$ is naturally extended, by density, to $\wunp$.
\end{itemize}

Then, following the construction in Section \ref{sec:cont.setting}, $\adiv$ is the standard divergence $\div_g$ on $M$ and $\hdiv=\{\bv\in L^2(TM)\,:\, \div_g\bv\in L^2(M)\}$.
We can then take $V=\mathrm{span}\{1\}$ and see that \eqref{eq:hypV} holds by the Poincar\'e--Wirtinger inequality in $\wunp$ (this inequality follows as in bounded open sets of $\R^d$ by using the compact embedding $\wunp\hookrightarrow L^2(M)$).

In the setting described by \eqref{hypgl}, Problem \eqref{eq:linpbstrong} contains as a particular case the Poisson equation $-\Delta_g \bu=f$ on $M$ (with selection of the unique solution  having zero average on the manifold), obtained by letting $\bfa(\nabla_g u) =\nabla_g u$ and $a(u) = \int_M u(x)\,d\mu_g(x)$. 
In its generic form, \eqref{eq:weakwunpp} is an extension of the Leray--Lions equations to $M$.

\bibliographystyle{abbrv}
\bibliography{gdm_abst}

\end{document}